\documentclass[12pt]{article}%
\usepackage{amssymb}
\usepackage{amsmath}
\usepackage{graphicx,color}
\usepackage{amsfonts}
\usepackage[ignoreall]{geometry}
\usepackage{graphicx}%
\providecommand{\U}[1]{\protect\rule{.1in}{.1in}}
\providecommand{\U}[1]{\protect\rule{.1in}{.1in}}

\newtheorem{theorem}{Theorem}

\newtheorem{definition}{Definition}

\newtheorem{lemma}{Lemma}

\newtheorem{proposition}{Proposition}
\newtheorem{remark}{Remark}

\newenvironment{proof}[1][Proof]{\noindent\textbf{#1.} }{\ \rule{0.5em}{0.5em}}
\begin{document}

\title{Efficient Simulation and Conditional Functional Limit Theorems for Ruinous
Heavy-tailed Random Walks}
\author{Jose Blanchet \thanks{Research supported in part by DMS-0902075 and
CMMI-0846816}\quad and Jingchen Liu \thanks{Research supported in part by
Institute of Education Sciences, U.S. Department of Education, through Grant
R305D100017}}
\maketitle

\begin{abstract}
The contribution of this paper is to introduce change of measure based
techniques for the rare-event analysis of heavy-tailed stochastic processes.
Our changes-of-measure are parameterized by a family of distributions admitting a mixture form. We exploit our methodology to achieve two types
of results. First, we construct Monte Carlo estimators that are strongly
efficient (i.e. have bounded relative mean squared error as the event of
interest becomes rare). These estimators are used to estimate both rare-event
probabilities of interest and associated conditional expectations. We
emphasize that our techniques allow us to control the expected termination
time of the Monte Carlo algorithm even if the conditional expected
stopping time (under the original distribution) given the event of interest is infinity -- a situation that
sometimes occurs in heavy-tailed settings. Second, the mixture family serves
as a good approximation (in total variation) of the conditional distribution
of the whole process given the rare event of interest. The convenient form of
the mixture family allows us to obtain, as a corollary, functional conditional
central limit theorems that extend classical results in the literature. We
illustrate our methodology in the context of the ruin probability $P(\sup
_{n}S_{n}>b)$, where $S_{n}$ is a random walk with heavy-tailed increments
that have negative drift. Our techniques are based on the use of Lyapunov
inequalities for variance control and termination time. The conditional limit
theorems combine the application of Lyapunov bounds with coupling arguments.


\end{abstract}

\section{Introduction}

Change-of-measure techniques constitute a cornerstone in the large
deviations analysis of stochastic processes (see for instance
\cite{DEMZEI98}). In the light-tailed setting, it is well understood that a
specific class of changes-of-measure, namely exponential tilting, provide
just the right vehicle to perform not only large deviations analysis but also
to design provably efficient importance sampling simulation estimators. There
is a wealth of literature on structural results, such as conditional limit
theorems, that justify the use of exponential changes of measure in these
settings (see for instance \cite{ASM82,ASMRUB95} in the setting of random walks and \cite{DupWang08} in the context of networks).

Our contribution in this paper is the introduction of change-of-measure techniques for the rare-event analysis of heavy-tailed stochastic processes.
Our general motivation is to put forward tools that allow to perform both,
large deviations analysis for heavy-tailed systems and, at the same time, construction of
efficient Monte Carlo algorithms for estimation of rare events, in the same
spirit as in light-tailed settings. To this end, we
introduce a family of changes of measures that are parameterized by a
mixture of finitely many distributions and develop mathematical tools for their analyses. We
concentrate on a class of problems of interest both in queueing theory and
risk theory, namely first passage time probabilities for random walks, which
serve as a good stylized model for testing and explaining
techniques at the interface of large deviations and simulation.
For instance, the first paper (\cite{SIE76}) that introduced the notations of efficiency together with the application of light-tailed large deviations ideas and exponential changes-of-measure, focused on this class of model problems. Such notations are now standard in rare-event simulation.
In the heavy-tailed setting, first passage time problems for random walks also serve as an environment for explaining the challenges that arise when
trying to develop efficient importance sampling estimators (see
\cite{ASMBINHOJ00}). We will provide additional discussion on those challenges
and contrast our methods here with recent approaches that have been developed
for first passage time problems for heavy-tailed random walks. We will illustrate the flexibility of our method in terms of simulation
estimators that have good variance performance and good control on the cost per
replication of the simulation estimator. The proposed change of measure also satisfies structural results
(in the form of conditional limit theorems) in the spirit of the theory that
has been developed in light-tailed environments. Let us introduce
the setup that will be the focus of our paper.

Let $S=\{S_{n}:n\geq0\}$ be a random walk with independently and identically
distributed (i.i.d.) increments, $\{X_{n}:n\geq1\}$, that is, $S_{n+1}%
=S_{n}+X_{n+1}$ for all $n\geq0$ and $S_{0}=0$. We assume that $\mu=EX_{n}<0$
and that the $X_{n}$'s are suitably heavy-tailed (see Section \ref{SectionMainResults}). For each $b\in\mathbb{R}%
^{+}$, let $\tau_{b}=\inf\{n\geq1:S_{n}>b\}$. Of interest in this paper is the
first passage time probability\footnote{If $S_{0}=0$ we use $P\left(  \cdot\right)  $ and
$E\left(  \cdot\right)  $ to denote the associated probability measure and
expectation operators in path space, respectively. If $S_{0}=s$, then
we write $P_{s}\left(  \cdot\right)  $ and $E_{s}\left(  \cdot\right)  $.}
\begin{equation}
u(b)=P(\tau_{b}<\infty), \label{Cd1}%
\end{equation}
and the conditional distribution of the random walk given $\{\tau_{b}%
<\infty\}$, namely
\begin{equation}
P(S\in\cdot|\tau_{b}<\infty). \label{Cd}%
\end{equation}

This paper introduces a family of unbiased simulation estimators for
$u\left(  b\right)  $ that can be shown to have bounded coefficient of
variation uniformly over $b>0$. The associated sampling distribution approximates (\ref{Cd}) in
total variation as $b\rightarrow\infty$. Unbiased estimators with
bounded coefficient of variation are called \emph{strongly efficient estimators}
in rare event simulation (Chapter 6 in \cite{ASMGLY07}).

The construction of provably efficient importance sampling estimators has been
the focus of many papers in the applied probability literature. A natural idea
behind the construction of efficient importance sampling estimators is that
one should mimic the behavior of the zero variance change of measure, which
coincides precisely with the conditional distribution (\ref{Cd}).
As it is well known, heavy-tailed large deviations are often
governed by the \textquotedblleft principle of the big jump", which,
qualitatively speaking, indicates that asymptotically as $b\rightarrow\infty$
the event of interest (in our case $\{\tau_{b}<\infty\}$) occurs due to the
contribution of a single large increment of size $\Omega(b)$.\footnote{For
$f(\cdot)$ and $g(\cdot)$ non-negative we use the notation $f(b)=O(g(b))$ if
$f(b)\leq cg(b)$ for some $c\in(0,\infty)$. Similarly, $f(b)=\Omega(g(b))$ if
$f(b)\geq cg(b)$ and we also write $f(b)=o(g(b))$ as $b\rightarrow\infty$ if
$f(b)/g(b)\rightarrow0$ as $b\rightarrow\infty$.} Consequently, the principle
of the big jump naturally suggests to mimic the zero variance
change of measure by a distribution which assigns zero probability to the event that ruin occurs due to the contribution
of more than one large jump of order $\Omega(b)$. However, such an importance sampling strategy is not feasible because it violates absolute continuity requirements to define a likelihood ratio. This is the
most obvious problem that arises in the construction of efficient importance
sampling schemes for heavy-tailed problems. A more subtle problem
discussed in \cite{ASMBINHOJ00} is the fact that the second moment of an
importance sampling estimator for heavy-tailed large deviations is often very
sensitive to the behavior of the likelihood ratio precisely on paths that
exhibit more than one large jump for the occurrence of the rare event in
question. We shall refer to those paths that require more than one large jump
for the occurrence of the event $\tau_{b}<\infty$ \emph{rogue paths}.

In the last few years state-dependent importance sampling has been used as a
viable way to construct estimators for heavy-tailed rare-event simulation. A natural idea is to exploit the
Markovian representation of (\ref{Cd}) in terms of the so-called Doob's
h-transform. In particular, it is well known that
\begin{equation}
P(X_{n+1}\in dx|S_{n},n<\tau_{b}<\infty)=\frac{u(b-S_{n}-x)}{u(b-S_{n}%
)}F\left(  dx\right)  , \label{Gold}%
\end{equation}
where $F$ is the distribution of $X_{n+1}$. In \cite{BlaGly07},
a state dependent importance sampling estimator based on an approximation to (\ref{Gold}) is constructed and a technique based on Lyapunov inequalities was introduced for variance control. In
particular, by constructing a suitable Lyapunov function, in \cite{BlaGly07},
it is shown that if $v(b-s)$ is a suitable approximation to $u(b-s)$ as
$b-s\nearrow\infty$ and $w(b-s)=Ev(b-s-X)$ then simulating the increment
$X_{n+1}$ given $S_{n}$ and $\tau_{b}>n$ via the distribution
\begin{equation}
\widetilde{P}\left(  X_{n+1}\in dx|S_{n}\right)  =\frac{v(b-S_{n}%
-x)}{w(b-S_{n})}F(dx) \label{BG08}%
\end{equation}
provides a strongly efficient estimator for $u\left(  b\right)  $. This
approach provided the first provably efficient estimator for $u\left(
b\right)  $ in the context of a general class of heavy-tailed increment
distributions, the class $S^{\ast}$,
which includes in particular Weibull and regularly varying distributions.
Despite the fact that the importance sampling strategy induced by (\ref{BG08})
has been proved to be efficient in substantial generality, it has a few
inconvenient features. First, it typically requires to numerically evaluate
$w\left(  b-S_{n}\right)  $ for each $S_{n}$ during the course of the
algorithm. Although this issue does not appear to be too critical in the one
dimensional setting (see the analysis in \cite{BLLI08}), for higher dimensional
problems, the numerical evaluation of $w\left(  b-S_{n}\right)  $
could easily require a significant computational overhead. For instance, see the first passage time computations for
multiserver queues, which have been studied in the regularly varying case in \cite{BlaGlyLiu07}.
The second
inconvenient feature is that if the increments have finite
mean but infinite variance we obtain $E\left(  \tau_{b}|\tau_{b}%
<\infty\right)  =\infty$. The strategy of mimicking the conditional
distribution without paying attention to the cost per replication of the
estimator could yield a poor overall computational complexity. Our proposed
approach does not suffer from this drawback because our parametric family of
changes of measures allows to control both the variance and the termination time.

We now proceed to explicitly summarize the contributions of this paper.
Further discussion will be given momentarily and precise mathematical
statements are given in Section \ref{SectSumResults}.

\begin{enumerate}
\item We provide a strongly efficient estimator (i.e. bounded relative mean
squared error as $b\nearrow\infty$) to compute the rare event probabilities
$u(b)$ and the associated conditional expectations, based on a finite
mixture family, for which both the simulation and density evaluation are
straightforward to perform (see Theorem \ref{ThmSE1}). Several features of the algorithm include:

\begin{enumerate}
\item The results require the distribution to have an eventually concave
cumulative hazard function, which includes a large class of distributions including
regularly varying, Weibull distribution, log-normal distribution and so forth
(see assumptions in Section \ref{SectionMainResults}).

\item One feature of the proposed algorithm relates to the
termination time. When the increments are regularly varying with tail index
$\iota\in(1,2)$, $E(\tau_{b}|\tau_{b}<\infty)=\infty$. This implies that the
zero-variance change of measure takes infinity expected time to generate one
sample. In contrast, we show that the proposed importance sampling algorithm
takes $O(b)$ expected time to generate one sample while still maintaining
strong efficiency if $\iota\in(1.5,2)$ --Theorem \ref{ThmSEb}.

\item For the case that $\iota\in(1,1.5]$, we show that the
$(1+\gamma)$-th moment of the estimator is of order $O(u^{1+\gamma}(b))$ with $\gamma>0$ depending on $\iota$. In addition, the expected termination time of the algorithm is $O(b)$ (Theorem \ref{Thm1_pls_Gamma}). Therefore, to compute $u(b)$ with
$\varepsilon$ relative error and at least $1-\delta$ probability, the total
computation complexity is $O(b)$.
\end{enumerate}

\item The mixture family approximates the conditional distribution of the
random walk given ruin in total variation. Based on this strong approximation
and on the simplicity of the mixture family's form we derive a conditional functional
central limit theorem of the random walk given ruin, which further
extends existing results reported in \cite{AK96} (compare Theorems
\ref{THAK}, \ref{ThmTVA} and \ref{ThmTVD} below).
\end{enumerate}

\bigskip

As mentioned earlier, the simulation estimators proposed in this paper are
based on importance sampling and they are designed to directly
mimic the conditional distribution of $S$ given $\tau_{b}<\infty$
based on the principle of the big jump. This principle suggests that one
should mimic the behavior of such a conditional distribution at each step by a
mixture of two components: one involving an increment distribution that is
conditioned to reach level $b$ and a second one corresponding to a nominal
(unconditional) increment distribution. This two-mixture sampler,
which was introduced by \cite{DULEWA06} in the context of tail estimation of a
fixed sum of heavy-tailed random variables, has been shown to produce strongly
efficient estimators for regularly varying distributions
\cite{DULEWA06,BGL07,BL08,BL10}. However, two-component
mixtures are not suitable for the design of strongly efficient estimators in
the context of other types of heavy-tailed distributions. In
particular, two-component mixtures are not applicable to 
semiexponential distributions (see \cite{BorMog06} for the definition)
such as Weibull.

As indicated, one of our main contributions in this paper is to introduce a
generalized finite-mixture sampler that can be shown to be suitable for
constructing strongly efficient estimators in the context of a general class
of heavy-tailed distributions, beyond regularly varying tails and including
lognormals and Weibullian-type tails. Our mixture family also mimics the
qualitative behavior mentioned above; namely, there is the contribution of a
large jump and the contribution of a regular jump. In addition, one needs to control the behavior of the likelihood ratio
corresponding to rogue sample paths. Depending on the degree
of concavity of the cumulative hazard function (which we assume to be eventually
strictly concave) we must interpolate between the large jump component and the
nominal component in a suitable way. At the end, the number of mixtures is
larger for cumulative hazard functions that are less concave.

Our mixture family and our Lyapunov based
analysis allow to obtain an importance sampling scheme that achieves strong
efficiency and \textit{controlled expected termination time even if the optimal (in terms
of variance minimization) change of measure involves an infinite expected
termination time}. More precisely, if the increment distribution is regularly
varying with tail index $\iota\in(1,2)$ it follows using the Pakes-Veraberbeke
theorem (see Theorem \ref{ThmPV}) that
\begin{align*}
E(\tau_{b}|\tau_{b}  &  <\infty)=\sum_{n=0}^{\infty}\frac{P\left(  \tau
_{b}>n,\tau_{b}<\infty\right)  }{P\left(  \tau_{b}<\infty\right)  }\\
&  \geq\sum_{n=1}^{\infty}\frac{P\left(  \tau_{b-\mu n/2}<\infty\right)
P\left(  |S_{n}+n\mu|\leq n\left\vert \mu\right\vert /2\right)  }{P\left(
\tau_{b}<\infty\right)  }=\infty.
\end{align*}
Nevertheless, as we will show, if $\iota\in(  1.5,2]  $ we can
choose the mixture parameters (which are state-dependent) in such a way that
(using $E^{Q}\left(  \cdot\right)  $ to denote the probability measure induced
by our importance sampling strategy assuming $S_{0}=0$)%
\begin{equation}
E^{Q}\tau_{b}=O\left(  b\right)  \label{ObExp}%
\end{equation}
while maintaining strong efficiency. We believe this feature is surprising!
In particular, it implies that one can construct a family of estimators for
expectations of the form $E(H(  S_{k}:k\leq\tau_{b}
)|\tau_{b}<\infty)$ that requires overall $O\left(  b\right)  $ random numbers
generated uniformly over a class of functions such that $0<K_{0}\leq H\leq
K_{1}<\infty$, even if $E(\tau_{b}|\tau_{b}<\infty)=\infty$. We shall also informally
explain why $\iota>1.5$ appears to be a necessary condition in
order to construct an unbiased estimator satisfying both strong efficiency and
(\ref{ObExp}).

In addition, for the case that $\iota\in(1,1.5]$, we are able to construct an
estimator whose $(1+\gamma)$-th moment (for $ 0 < \gamma < (\iota-1)/(2-\iota)$) is of order $O(u^{1+\gamma}(b))$ while
the expected termination time is $O(b)$. We will also argue that the bound on $\gamma$ is essentially optimal. Consequently, as it is shown in Theorem \ref{Thm1_pls_Gamma}, to
compute $u(b)$ with $\varepsilon$ relative error and at least $1-\delta$
probability, the total computational complexity is $O(b)$.

In addition to providing a family of strongly efficient estimators for
$u\left(  b\right)  $, our finite-mixture family can
approximate the conditional measure (\ref{Cd}) in total variation as
$b\nearrow\infty$. This approximation step further strengthens
our family of samplers as a natural rare-event simulation scheme for heavy-tailed systems.
Moreover, given the strong mode of convergence and
because the mixture family admits a friendly form, we are able to strengthen
classical results in the literature on heavy tailed approximations, see
\cite{AK96}. For instance, if a given increment has second moment, we will
derive, as a corollary of our approximations, a conditional functional central
limit theorem up to the first passage time $\tau_{b}$. Thereby, this improves
the law of large numbers derived in \cite{AK96}. Another related result in the
setting of high dimension regularly varying random walk is given in
\cite{HLMS05}. We believe that the proof techniques behind our approximations,
which are based on coupling arguments, are of independent interest and that
they can be used in other heavy-tailed environments.

A central technique in the analysis of both the computational complexity and our conditional limit theorems is the use of
Lyapunov functions. The Lyapunov functions are used for three different
purposes: First in showing the strong efficiency of the importance sampling
estimator, second in providing a bound on the finite expected termination time
of the algorithm, and finally in proving the approximation in total variation
of the zero-variance change of measure. The construction of Lyapunov functions
follows the so called fluid heuristic, which is well known in the literature
of heavy-tailed large deviations and has also been successfully applied in
rare event simulation, see \cite{BGL07,BL08,BL10}.

This paper is organized as follows. In Section \ref{SectionMainResults}, we
introduce our assumptions, our family of changes of measures and we provide
precise mathematical statements of our main results. Section \ref{SecPre}
discusses some background results on large deviations and Lyapunov
inequalities for importance sampling and stability of Markov processes. The
variance analysis of our estimators is given in Section \ref{SectionAnalysis}.
The results corresponding to the termination time of our algorithm can be
found in Section \ref{SecTermination}. Then we have our results on strong
conditional limit theorems in Section \ref{SecTV}. We provide numerical
experiments in Section \ref{SecSim}. Finally, we added an appendix which
contains auxiliary lemmas and technical results.

\section{Main Results\label{SectionMainResults}}

We shall use $X$ to denote a generic random variable with the same
distribution as any of the $X_{i}$'s describing the random walk $S_{n}
=\sum_{i=1}^{n}X_{i}$, for $n=1,2,...$ with $S_{0}=0$. We write $F(x)=P(X\leq x)$, $\bar{F}(x)=P(X>x)$ and $EX=\mu\in\left(  -\infty,0\right)  $. Further, let
$\Lambda(\cdot)$ be the cumulative hazard function and $\lambda(\cdot)$ be the hazard
function. Therefore, $F$ has density function, for $x\in\left(  -\infty
,\infty\right)  $
\[
f(x)=\lambda(x)e^{-\Lambda(x)},\quad\mbox{and }\bar{F}(x)=e^{-\Lambda(x)}.
\]
Of primary interest to us is the design of efficient importance sampling
(change of measure based) estimators for%
\begin{equation}
u(b)=P(\max_{n\geq1}S_{n}>b)=P(\tau_{b}<\infty), \label{u}%
\end{equation}
as $b\rightarrow\infty$ when $F$ is suitably heavy-tailed. In
particular, throughout this paper we shall assume either of the following two
sets of conditions:

\textbf{Assumption A:} $F$ has a regularly varying right tail with index
$\iota>1$. That is,
\[
\bar{F}(x)=1-F(x)=L(x)x^{-\iota},
\]
where $L(\cdot)$ is a slowly varying function at infinity, that is, $\lim_{x\rightarrow\infty}L(xt)/L(x)=1$ for all $t\in (0,1]$.

\bigskip

Or

\bigskip

\textbf{Assumption B:} There exists $b_{0}>0$ such that for all $x\geq b_{0}$
the following conditions hold.

\begin{itemize}
\item[B1] Suppose that $\lim_{x\rightarrow\infty}x\lambda(x)=\infty$.

\item[B2] There exists $\beta_{0}\in(0,1)$ such that $\partial\log
\Lambda(x)=\lambda\left(  x\right)  /\Lambda\left(  x\right)  \leq\beta
_{0}x^{-1}$ for $x\geq b_{0}$.

\item[B3] Assume that $\Lambda\left(  \cdot\right)  $ is concave for all
$x\geq b_{0}$; equivalently, $\lambda\left(  \cdot\right)  $ is assumed to be
non increasing for $x\geq b_{0}$.

\item[B4] Assume that
\[
P\left(  X>x+t/\lambda\left(  x\right)  |X>x\right)  =\exp\left(  -t\right)
(1+o\left(  1\right)  )
\]
as $x\nearrow\infty$ uniformly over compact sets in $t\geq0$. In addition, for
some $\alpha>1$, $P(X>x+t/\lambda(x)|X>x)\leq t^{-\alpha}$ for all $t,x>b_{0}$.
\end{itemize}

\begin{remark}
The analysis requires $\Lambda\left(  \cdot\right)  $ to be differentiable
only for $x\geq b_{0}$. The reason for introducing Assumptions A and B
separately is that the analysis for regularly varying distributions is
somewhat different from (easier than) the cases under Assumption B. Assumption B1 implies that the tail of $X$ decays
faster than any polynomial. Assumptions B2 and B3 basically say that the
cumulative hazard function of $F$ is ``more concave''
than at least some Weibull distribution with shape parameter $\beta_{0}<1$.
Typically, the more concave the cumulative hazard function is, the heavier the tail is.
Therefore, under Assumption B, $F$ is basically assumed to have a heavier tail
than at least some Weibull distribution with shape parameter $\beta_{0}<1$.
Assumption B4 is required only in Theorem \ref{ThmTVD} which states the
functional central limit theorem of the conditional random walk given ruin.
Note that the Assumptions A and B cover a wide range of heavy-tailed distributions that
are popular in practice, for instance, regularly varying, log-normal, Weibull
with $\beta_{0}\in(0,1)$ and so forth.
\end{remark}

In our random walk context, state-dependent importance sampling involves
studying a family of densities (depending on ``current'' state $s$ of the random walk)
which governs subsequent increments of the random walk. More precisely, we
write
\[
q_{s}\left(  x\right)  =r_{s}\left(  x\right)  ^{-1}f\left(  x\right)  ,
\]
where $r_{s}(\cdot)$ is a non-negative function such that $Er_{s}\left(
X\right)  =1$ for a generic family of state-dependent importance sampling
increment distributions. If we let $Q\left(  \cdot\right)  $ represent the
probability measure in path-space induced by the subsequent generation of
increments under $q_{s}\left(  \cdot\right)  $, then it follows easily that%
\[
u\left(  b\right)  =E^{Q}[I\left(  \tau_{b}<\infty\right)  L_{b}],
\]
with
\begin{equation}
L_{b}=\sum_{j=1}^{\tau_{b}}r_{S_{j-1}}\left(  S_{j}-S_{j-1}\right)  .
\label{LR1}%
\end{equation}
We say that
\begin{equation}
Z_{b}=I\left(  \tau_{b}<\infty\right)  L_{b} \label{IS1}%
\end{equation}
is an importance sampling estimator for $u\left(  b\right)  $ and its second
moment is simply
\[
E^{Q}[I\left(  \tau_{b}<\infty\right)  L_{b}^{2}]=E[I\left(  \tau_{b}%
<\infty\right)  L_{b}].
\]
If we select $Q\left(  \cdot\right)  =P\left(
\cdot|\tau_{b}<\infty\right)  $, or equivalently we let $r_{s}\left(
x\right)  =u(b-s)/u\left(  b-s-x\right)  $, then the corresponding importance
sampling estimator would yield zero variance. Hence, we call it zero-variance
importance sampling estimator; and we call $P\left(  \cdot|\tau_{b}%
<\infty\right)  $ the zero-variance change of measure or zero-variance
importance sampling distribution.

One of our main goals in this paper is to show that we can approximate the
zero-variance change of measure quite accurately using finitely many mixtures
whose parameters can be easily computed in advance. As a consequence, we can
use Monte Carlo simulation to not only accurately estimate $u\left(  b\right)
$ but also associated conditional expectations of the random walk given
$\tau_{b}<\infty.$ In fact, we can improve upon the zero variance change of measure in terms of overall computational cost when it
comes to estimating sample-path conditional expectations given $\tau
_{b}<\infty$ in situations where $E\left(  \tau_{b}|\tau_{b}<\infty\right)
=\infty$. The precise mathematical statements are given later in this section. Future sections are dedicated to the
development and the proofs of these statements.

Before stating the main results, we would first introduce the family of change
of measures which is based on a mixture of finitely many computable and
simulatable distributions.

\subsection{The mixture family\label{SubMixFam}}

We start by describing the precise form of the mixtures that we will use to
construct efficient importance sampling schemes. The family is constructed to consider the contribution of a ``large jump'' which makes the walk reach level $b$ in the next step,
a ``regular jump" which allows the random walk to continue
under (nearly) its original dynamics, and a number of ``interpolating" contributions. This intuition is consistent with the way in
which large deviations occur in heavy-tailed environments.

If $b-s>\eta_{\ast}$ for $\eta_{\ast}>0$ sufficiently large and to be
specified in our analysis, we propose to use a finite mixture family of the
form%
\begin{equation}
q_{s}(x)=p_{\ast}f_{\ast}(x|s)+p_{\ast\ast}f_{\ast\ast}(x|s)+\sum_{j=1}%
^{k}p_{j}f_{j}(x|s), \label{SampDist}%
\end{equation}
where $p_{\ast}$, $p_{\ast\ast}$, $p_{j}\in\lbrack0,1)$, $p_{\ast}+p_{\ast
\ast}+\sum_{j=1}^{k}p_{j}=1$, $k\in\mathbb{N}$, and $f_{\ast}$, $f_{\ast\ast}%
$, and $f_{j}$ for $j=1,..,k$ are properly normalized density functions, whose
supports are disjoint and depend on the \textquotedblleft
current\textquotedblright\ position of the walk, $s$. We will give specific
forms momentarily. The choice of $k$ depends on the
concavity of the cumulative hazard function, but otherwise is independent of $b$ and $s$.
We will ultimately let $p_{\ast},p_{\ast\ast}$ and the $p_{j}$'s depend
on $s$. In addition, we will also choose \textit{not to apply importance
sampling} if we are suitably close to the boundary level $b$. In other words, overall
we have that
\begin{equation}
q_{s}\left(  x\right)  =\Big[p_{\ast}f_{\ast}(x|s)+p_{\ast\ast}f_{\ast\ast
}(x|s)+\sum_{j=1}^{k}p_{j}f_{j}(x|s)\Big]I\left(  b-s>\eta_{\ast}\right)
+f\left(  x\right)  I\left(  b-s\leq\eta_{\ast}\right)  . \label{Sel_q}%
\end{equation}

We next specify the functional forms of each mixture distribution. First,%
\[
f_{\ast}(x|s)=f(x)\frac{I(x\leq b-s-\Lambda^{-1}(\Lambda(b-s)-a_{\ast}%
))}{P(X\leq b-s-\Lambda^{-1}(\Lambda(b-s)-a_{\ast}))},
\]
where $a_{\ast}>0$. So, $f_{\ast}$ represents the mixture component
corresponding to a \textquotedblleft regular\textquotedblright\ increment.

Further, for $a_{\ast\ast}>0$, let
\[
f_{\ast\ast}(x|s)=f(x)\frac{I(x>\Lambda^{-1}(\Lambda(b-s)-a_{\ast\ast}%
))}{P(X>\Lambda^{-1}(\Lambda(b-s)-a_{\ast\ast}))}.
\]
$f_{\ast\ast}$ represents the mixture component corresponding to the situation in which the rare event occurs
because this particular increment is large. Note that%
\[
P\left(  X>b-s|X>\Lambda^{-1}(\Lambda(b-s)-a_{\ast\ast})\right)  =\exp\left(
-a_{\ast\ast}\right)  .
\]
Therefore, if the \textquotedblleft next increment\textquotedblright, $X$,
given the current position, $s$, is drawn from $f_{\ast\ast}$, there is
probability $1-\exp\left(  -a_{\ast\ast}\right)  >0$ that the next position of
the random walk, namely $s+X$, is below the threshold $b$. This particular
feature is important in the variance control. It is necessary to introduce
such a positive $a_{\ast\ast}$ to achieve strong efficiency if we want to
consider the possibility of rogue paths in our sampler.

As we mentioned before, the choice of $k$ depends on the \textquotedblleft
concavity\textquotedblright\ of the cumulative hazard function $\Lambda(\cdot)$. The more
concave $\Lambda(\cdot)$ is, the smaller $k$ one can usually choose. In the
regularly varying case, for example, a two-mixture distribution is sufficient
(i.e. $k=0$). The analysis of importance sampling algorithms in this case
has been substantially studied in the literature (see
\cite{DULEWA06,BGL07,BL08,BL10}). We can see that this feature is captured in our
current formulation because in the regularly varying case one can find
$a_{\ast},a_{\ast\ast}>0$ such that
\begin{equation}
b-s-\Lambda^{-1}(\Lambda(b-s)-a_{\ast})\geq\Lambda^{-1}(\Lambda(b-s)-a_{\ast
\ast}), \label{k0}%
\end{equation}
for all $b-s$ large enough so that one can choose $k=0$. Indeed, to see how
(\ref{k0}) holds for the regularly varying case, just note that for any
$a\in\left(  0,1\right)  $, for each $t$, the inequality
\[
at\geq\Lambda^{-1}(\Lambda\left(  t\right)  -a_{\ast\ast})
\]
is equivalent to%
\begin{equation}
\frac{P\left(  X>at\right)  }{P\left(  X>t\right)  }\leq\exp\left(
a_{\ast\ast}\right)  . \label{RVRed1}%
\end{equation}
Similarly,
\[
t-\Lambda^{-1}(\Lambda(t)-a_{\ast})\geq at
\]
holds if and only if%
\begin{equation}
\frac{P\left(  X>(1-a)t\right)  }{P\left(  X>t\right)  }\leq\exp\left(
a_{\ast}\right)  . \label{RVRed2}%
\end{equation}
Karamata's theorem for regularly varying distributions ensures that it is
always possible to choose $a_{\ast},a_{\ast\ast}>0$ given any $a\in\left(
0,1\right)  $ so that (\ref{RVRed1}) and (\ref{RVRed2}) hold for uniformly in
$t$ and therefore we have that (\ref{k0}) holds. If Assumption A holds, we choose $a_{**}$ and then select $a_{*}$ (possibly depending on $b-s$) such that
\begin{equation}\label{Reg}
b-s-\Lambda^{-1}(\Lambda(b-s)-a_{\ast})=\Lambda^{-1}(\Lambda(b-s)-a_{\ast
\ast}).
\end{equation}
This selection is slightly different from the two-mixture form that has been analyzed in the literature (see \cite{BGL07,BL08,BL10}) which involves a ``regular'' component with support on $(-\infty,a(b-s)]$ and a ``large jump'' component with support on $(a(b-s),\infty)$, for $a\in(0,1)$. Our analysis here also applies to this parameterization. Nevertheless, to have unified statements in our results, under both Assumptions A and B, we opted for using equation \eqref{Reg}.

When (\ref{k0}) does not hold (for instance in the case of Weibull tails with
shape parameter $\beta\in(0,1)$), we will need more mixtures. In particular,
we consider a set of cut-off points $c_{0}<...<c_{k}$ depending on $b-s$.
Ultimately, we will have%
\[
c_{j}=a_{j}\left(  b-s\right)  \text{ \ for }j=1,2,...,k-1\text{.}%
\]
where $a_{1}<...<a_{k-1}$. The $a_{j}$'s are precomputed depending on
$\beta_{0}$ (from Assumption B3) according to Lemma \ref{LemWeibull} (Section \ref{SectionAnalysis}). We
let $c_{0}=b-s-\Lambda^{-1}(\Lambda(b-s)-a_{\ast}))$ and $c_{k}=\Lambda
^{-1}(\Lambda(b-s)-a_{\ast\ast}))$. Given these values we define for $1\leq
j\leq k-1$,
\[
f_{j}(x)=f(x)\frac{I(x\in(c_{j-1},c_{j}])}{P(X\in(c_{j-1},c_{j}])}.
\]
For $j=k$,
\[
f_{k}(x)=f(b-s-x)\frac{I(x\in(c_{k-1},c_{k}])}{P(X\in(b-s-c_{k},b-s-c_{k-1}%
])}.
\]
In our previous notation, we then can write%
\begin{align*}
&  r_{s}\left(  x\right)  ^{-1}\\
&  =\left(  \frac{p_{\ast}I(x\leq c_{0})}{P(X\leq c_{0})}+\frac{p_{\ast\ast
}I(x>c_{k})}{P(X>c_{k})}+\sum_{j=1}^{k-1}\frac{p_{j}I(x\in(c_{j-1},c_{j}%
])}{P(X\in(c_{j-1},c_{j}])}+\frac{f(b-s-x)p_{k}I(x\in(c_{k-1},c_{k}%
])}{f(x)P(X\in(b-s-c_{k},b-s-c_{k-1}])}\right) \\
&  \times I\left(  b-s>\eta_{\ast}\right)  +I\left(  b-s\leq\eta_{\ast
}\right)  .
\end{align*}

With this family of change of measures, we are ready to present our main
results which are based on appropriate choices of the various tuning parameters.

\subsection{Summary of the results\label{SectSumResults}}

Our first result establishes that one can explicitly choose $\eta_{*}$,
$c_{j}$'s, $a_{\ast}$, $a_{\ast\ast}$, $p_{\ast}$, $p_{\ast\ast}$ and the
$p_{j}$'s in order to have a strongly efficient (in the terminology of
rare-event simulation, see \cite{ASGLYNN}) estimator.

\begin{theorem}
\label{ThmSE1}Under either Assumptions A or B1-3, there exists an explicit
selection of $\eta_{*}$, the $c_{j}$'s, $a_{\ast}$, $a_{\ast\ast}$, $p_{\ast}%
$, $p_{\ast\ast}$ and the $p_{j}$'s so that the estimator $Z_{b}$ (defined as
in (\ref{IS1})) is strongly efficient in the sense of being unbiased and
having a bounded coefficient of variation. In particular, one can compute
$K\in\left(  0,\infty\right)  $ (uniform in $b>0$) such that%
\[
\frac{E^{Q}Z_{b}^{2}}{\left(  E^QZ_{b}\right)  ^{2}}=\frac{EL_{b}I\left(
\tau_{b}<\infty\right)  }{u\left(  b\right)  ^{2}}<K
\]
for $b>0$.
\end{theorem}

The proof of this result is given at the end of Section \ref{SectionAnalysis}.
The explicit parameter selection is discussed in items I) to IV) stated in
Section \ref{SectionAnalysis}. A consequence of this result is that, by
Chebyshev's inequality, at most $n=O\left(  \varepsilon^{-2}\delta
^{-1}\right)  $ i.i.d. replications of $Z_{b}$ are enough in order to estimate
$u\left(  b\right)  $ with $\varepsilon$-relative precision and with
probability at least $1-\delta$ uniformly in $b$. Because the estimator
$Z_{b}$ is based on importance sampling, one can estimate a large class of
expectations of the form $u_{H}\left(  b\right)  =E(H\left(  S_{n}:n\leq
\tau_{b}\right)  |\tau_{b}<\infty)$ with roughly the same number of
replications in order to achieve $\varepsilon$-relative precision with at
least $1-\delta$ probability (uniformly in $b$). Indeed, if $K_{1}\in\left(
0,\infty\right)  $ is such that $K_{1}^{-1}\leq H\leq K_{1}$ then we have that
$u_{H}\left(  b\right)  \geq K_{1}^{-1}$. We also have that $L_{b}I\left(
\tau_{b}<\infty\right)  H\left(  S_{n}:n\leq\tau_{b}\right)  $ is an unbiased
estimator for $E(H\left(  S_{n}:n\leq\tau_{b}\right)  ;\tau_{b}<\infty)$ and
its second moment is bounded by $K_{1}^{2}u\left(  b\right)  ^{2}$. Therefore,
we can estimate both the numerator and the denominator in the expression
\[
u_{H}(b)=E\left(  H\left(  S_{n}:n\leq\tau_{b}\right)  \ |\ \tau_{b}%
<\infty\right)  =\frac{E(H\left(  S_{n}:n\leq\tau_{b}\right)  ;\tau_{b}%
<\infty)}{u\left(  b\right)  }%
\]
with good relative precision (uniformly in $b$). Naturally, the condition
$K_{1}^{-1}\leq H\leq K_{1}$ is just given to quickly explain the significance
of the previous observation. More generally, one might expect strong
efficiency for $u_{H}\left(  b\right)  $ using an importance sampling
estimator designed to estimate $u\left(  b\right)  $ if $u_{H}\left(
b\right)  \in(K_{1}^{-1},K_{1})$ uniformly in $b$.

Given that nothing has been said about the cost of generating a single
replication of $Z_{b}$, strong efficiency is clearly not a concept that allows
to accurately assess the total computational cost of estimating $u\left(
b\right)  $ or $u_{H}\left(  b\right)  $. For this reason, we will also
provide results that estimate the expected cost required to generate a single
replication of $Z_{b}$. However, before we state our estimates for the cost
per replication, it is worth discussing what is the performance of the
zero-variance change of measure for the regularly varying case. The following
classical result (\cite{AsmKlup96}) provides a good description of $\left(
S_{n} :n\geq0\right)  $ given $\tau_{b}<\infty$.

\begin{theorem}
[Asmussen and Kluppelberg]\label{THAK}Suppose that $X$ is regularly varying
with index $\iota>1$ and define $a\left(  b\right)  =\int_{b}^{\infty}P\left(
X>u\right)  du/P\left(  X>b\right)  $. Then, conditional on $\tau_{b}<\infty$
we have that%
\[
\left(  \frac{\tau_{b}}{a\left(  b\right)  },\left(  \frac{S_{\left\lfloor
u\tau_{b}\right\rfloor }}{\tau_{b}}:0\leq u<1\right)  ,\frac{S_{\tau_{b}}%
-b}{b}\right)  \Longrightarrow\left(  Y_{0}/\left\vert \mu\right\vert ,\left(
u\mu:0\leq u<1\right)  ,Y_{1}\right)  ,
\]
where the convergence occurs in the space $R\times D[0,1)\times R$, $P\left(
Y_{i}>t\right)  =(1+t/(\iota-1))^{-\iota+1}$ for $t\geq0$ and $i=0,1$ and
$P\left(  Y_{0}>y_{0},Y_{1}>y_{1}\right)  =P\left(  Y_{0}>y_{0}+y_{1}\right)
.$
\end{theorem}

\begin{remark}
The previous result suggests that if Assumption A holds, the best possible
performance that one might realistically expect is $E^{Q}\tau_{b}=O\left(
b\right)  $ as long as (very important!) $\iota>2$. The full statement of
Asmussen and Kluppelberg's result (Theorem 1.1 in \cite{AsmKlup96}) also covers other subexponential
distributions. For instance, in the case of Weibull-type tails with shape
parameter $\beta_{0}$, their result suggests that $E(\tau_{b}|\tau_{b}%
<\infty)=O(b^{1-\beta_{0}})$.
\end{remark}

As the next theorem states, for the regularly varying case with $\iota>1.5$, we can guarantee $E^{Q}\tau_{b}=O\left(
b\right)  $ while maintaining strong efficiency as stated in Theorem
\ref{ThmSE1}. We will also indicate why we believe that this result is
basically the best possible that can be obtained among a reasonable class of
importance sampling distributions.

\bigskip

\begin{theorem}
\label{ThmSEb} \

\begin{itemize}
\item If Assumption A holds and $\iota>1.5$, then there exists an explicit
selection of $\eta_{*}$, the $c_{j}$'s, $a_{\ast}$, $a_{\ast\ast}$, $p_{\ast}%
$, $p_{\ast\ast}$ such that strong efficiency (as indicated in Theorem
\ref{ThmSE1}) holds and%
\[
E^{Q}\tau_{b}\leq\rho_{0}+\rho_{1}b
\]
for some $\rho_{0},\rho_{1}>0$ independent of $b$.

\item If Assumptions B1-3 hold, we assume there exists $\delta>0$ and
$\beta\in[0,\beta_{0}]$ such that $\lambda(x) \geq\delta x^{\beta-1}$ for $x$
sufficiently large. Then, with the parameters selected in Theorem \ref{ThmSE1}, there exists $\rho_{0}$ and $\rho_{1}$ independent
of $b$, such that,
\[
E^{Q}\tau_{b}\leq\rho_{0}+\rho_{1}b^{1-\beta}.
\]

\end{itemize}
\end{theorem}

\begin{remark}
The results in this theorem follow directly as a consequence of Propositions
\ref{PropReg} and \ref{PropWeibull} in Section \ref{SecTermination}. For the
regularly varying case (Assumption A), in addition to the explicit parameter
selection indicated in items I) to IV) in Section \ref{SectionAnalysis}, which
guarantee strong efficiency, we also add item V) in Section
\ref{SecTermination}, which explicitly indicates how to select the parameters
to obtain $O\left(  b\right)  $ expected stopping time while maintaining
strong efficiency. We assume that it takes at most a fixed cost $c$ of
computer time units to generate a variable from $q_{s}\left(  \cdot\right)  $
(uniformly in $s$). The previous result implies that if $X$
is regularly varying with index $\iota>1.5$, then our importance sampling
family estimates $u\left(  b\right)  $ and associated conditional expectations
such as $u_{H}\left(  b\right)  $ in $O\left(  \varepsilon^{-2}\delta
^{-1}b\right)  $ units of computer time. This is in some sense (given that we
have linear complexity in $b$ even if $\iota\in(1.5,2)$) better than what one
might expect in view of Theorem \ref{THAK}. We will further provide an
argument, see Remark \ref{RemOp3_2} in Section \ref{SecTermination}, for why
in the presence of regular variation $\iota>3/2$ appears to be basically a
necessary condition to obtain strongly efficient unbiased estimators with
$O(b)$ expected termination time.
\end{remark}

\begin{remark}
For the second case in Theorem \ref{ThmSEb}, note that when Assumption B1
holds, one can always choose $\beta=0$ and $\delta$ arbitrarily large.
This implies that the expected termination time is at the most $O(b)$ under Assumption B. It is desirable to choose $\beta$ as large as possible because this yields a (asymptotically) smaller termination time. However, there is an upper bound, namely $\beta_{0}$,
which can be derived from Assumption B2 (Lemma \ref{LemRateFun}).
\end{remark}

For the regularly varying case, we provide further results for all $\iota>1$.
If $\iota>1$, we are able to construct an importance sampling estimator $Z_{b}$
such that for some $\gamma>0$ we can guarantee $E^{Q}(Z_{b}^{1+\gamma})\leq
Ku\left(  b\right)  ^{1+\gamma}$ and at the same time $E^{Q}\tau_{b}=O\left(
b\right)  $. The next result, whose proof is given at the end of Section
\ref{SecTermination}, allows us to conclude that this can be achieved with our
method as well.

\begin{theorem}
\label{Thm1_pls_Gamma}Suppose that Assumption A is in force and $\iota
\in(1,1.5]$. Then, for each $\gamma\in(0,(\iota-1)/(2-\iota))$ we can select $K>0$, and a
member of our family of importance sampling distributions such that
\[
E^{Q}(Z_{b}^{1+\gamma})\leq Ku\left(  b\right)  ^{1+\gamma}%
\]
for all $b>0$ and $E^{Q}\left(  \tau_{b}\right)  \leq\rho_{0}+\rho_{1}b$ for
$\rho_{0},\rho_{1}\in\left(  0,\infty\right)  $. Consequently, assuming that
each increment under $q_{s}\left(  \cdot\right)  $ takes at most constant
units of computer time, then $O\left(  \varepsilon^{-2/\gamma}%
\delta^{-1/\gamma}b\right)  $ expected total cost is required to obtain an
estimate for $u\left(  b\right)  $ with $\varepsilon$ relative error and with
probability at least $1-\delta$.
\end{theorem}

\begin{remark}
Similar to the case of controlling the second moment, we believe that the upper bound $(\iota-1)/(2-\iota)$ is optimal within a reasonable class of simulation algorithms. A heuristic argument will be given in Section \ref{SecTermination}.

\end{remark}

\bigskip

Finally, the proposed family of change of measures and analysis techniques are useful not only for Monte Carlo simulation purposes but also for
asymptotic analysis. We provide the following approximation results which
improve upon classical results in the literature such as Theorem \ref{THAK}.
By appropriately tuning various parameters in our
family we can approximate
$P(S\in\cdot|\tau_{b}<\infty)$ by $Q\left(  S\in\cdot\right)  $ asymptotically
as $b\nearrow\infty$. We will explicitly indicate how to do so in later analysis.

\bigskip

\begin{theorem}
\label{ThmTVA}Under either Assumptions A or B1-3, there exists an explicit
selection of $\eta_{*}$, the $c_{j}$'s, $a_{\ast}$, $a_{\ast\ast}$, $p_{\ast}%
$, $p_{\ast\ast}$ and the $p_{j}$'s so that
\[
\overline{\lim}_{b\rightarrow\infty}\sup_{A}\left\vert P(S\in A|\tau
_{b}<\infty)-Q\left(  S\in A\right)  \right\vert =0.
\]

\end{theorem}

The previous result is an immediate consequence of Lemma \ref{LemTV1} combined
with Theorem \ref{ThmTV}. It further shows that our mixture family is an
appropriate vehicle to approximate the conditional distribution of the random
walk given $\tau_{b}<\infty$. Moreover, due to the convenience of the mixture
form, as a corollary of the previous theorem and using a coupling technique,
we can show, without much additional effort, the following theorem which
further extends Theorem 1.1 in \cite{AsmKlup96} by adding a central limit
theorem correction term. This theorem is proven at the end of Section
\ref{SectCLT}.




\begin{theorem}
\label{ThmTVD}Suppose that either Assumption A or Assumptions B1-4 are
in force. Let $\sigma^{2}=Var\left(  X_{1}\right) < \infty $ and $a\left(
b\right)  =\int_{b}^{\infty}P\left(  X>u\right)  du/P\left(  X>b\right)  $.
Then%
\[
\left(  \frac{\tau_{b}}{a(b)},\left\{  \frac{S_{\left[  t\tau_{b}\right]
}-t\mu\tau_{b}}{\sqrt{\tau_{b}}}\right\}  _{0\leq t<1},\frac{S_{\tau_{b}}%
-b}{a(b)}\right)  \Rightarrow(Y_{0}/\left\vert \mu\right\vert ,\left\{  \sigma
B(t)\right\}  _{0\leq t<1},Y_{1}),
\]
in $R\times D[0,1)\times R$. $\{B(t): 0\leq t<1\}$ is a standard Brownian motion independent of $(Y_0, Y_1)$. The joint law of $Y_{0}$ and $Y_{1}$ is
defined as follows. First, $P\left(  Y_{0}>y_0,Y_{1}>y_1\right)  =P\left(
Y_{1}>y_0+y_1\right)  $ with $Y_{0}\overset{d}{=}Y_{1}$ and

\begin{itemize}
\item If Assumption A holds then
\[
P\left(  Y_{1}>t\right)  =\frac{1}{(1+t/(\iota-1))^{\iota-1}}.
\]

\item If Assumptions B1-4 hold, then $Y_{1}$ follows exponential distribution with mean 1 and consequently $Y_0$ and $Y_1$ are independent.
\end{itemize}
\end{theorem}

\section{Preliminaries: Heavy tails, importance sampling and Lyapunov
inequalities}

\label{SecPre}

\subsection{Heavy tails}

A non-negative random variable $Y$ is said to be heavy-tailed if $E\exp\left(
\theta Y\right)  =\infty$ for every $\theta>0$. This class is too big to
develop a satisfactory asymptotic theory of large deviations and therefore one
often considers the subexponential distributions which are defined
as follows.

\begin{definition}
\label{DefSubExp} Let $Y_{1},...,Y_{n}$ be independent copies of a
non-negative random variable $Y$. The distribution of $Y$ (or Y itself) is
said to be subexponential if and only if
\[
\lim_{u\rightarrow\infty}\frac{P\left(  Y_{1}+...+Y_{n}>u\right)  }{P\left(
Y>u\right)  }=n.
\]
Actually it is necessary and sufficient to verify the previous limit for $n=2$ only.
\end{definition}

Examples of distributions that satisfy the subexponential property include
Pareto distribution, Lognormal distributions, Weibull distributions, and so
forth. A general random variable $X$ is said to have a subexponential right
tail if $X^{+}$ is subexponential. In such a case, we simply say that $X$ is subexponential.

If $X$ is subexponential, then $X$ satisfies that $P\left(  X>x+h\right)
/P\left(  X>x\right)  \rightarrow1$ as $x\rightarrow\infty$ for each
$h\in\left(  -\infty,\infty\right)  $. A random variable with this property is
said to possess a ``long tail". It turns out that there are long tailed random
variables that do not satisfy the subexponential property (see \cite{EKM97}).

In order to verify the subexponential property in the context of random
variables with a density function (as we shall assume here) one often takes
advantage of the so-called cumulative hazard function. Indeed, a sufficient condition to
guarantee subexponentiality due to Pitman is given next (see \cite{EKM97}).

\begin{proposition}
\label{PropPitmanC} A random variable $X$ with concave cumulative hazard function
$\Lambda\left(  \cdot\right)  $ and hazard function $\lambda\left(
\cdot\right)  $ is subexponential if%
\[
\int_{0}^{\infty}\exp\left(  x\lambda\left(  x\right)  -\Lambda\left(
x\right)  \right)  dx<\infty.
\]

\end{proposition}

\bigskip

A distinctive feature of heavy-tailed random walks is that the rare event
$\{\sup_{n}S_{n}>b\}$ is asymptotically (as $b\rightarrow\infty$) caused by a
single large increment, while other increments behave like ``regular'' ones.
Therefore, one can obtain the following approximation, often called
fluid heuristic, for the probability $u(b)$:
\begin{align}
u(b)  &  =P(\tau_{b}<\infty)=\sum_{k=1}^{\infty}P(\tau_{b}=k)\label{fluid}\\
&  \approx\sum_{k=1}^{\infty}P(X_{k}>b-(k-1)\mu)\approx-\frac{1}{\mu}\int
_{b}^{\infty}P(X>s)ds.\nonumber
\end{align}
For notational convenience, we denote the integrated tail by
\begin{equation}
G(x)=\int_{x}^{\infty}P(X>s)ds. \label{IntTail}%
\end{equation}
The previous heuristic can actually be made rigorous under subexponential
assumptions. This is the content of the Pakes-Veraberbeke theorem which we
state next (see page 296 in \cite{ASM03}).

\begin{theorem}
[Pakes-Veraberbeke]\label{ThmPV}If $F$ is long tailed (i.e. $\bar{F}(x+h)
/\bar{F}(x) \longrightarrow1$ as $x\nearrow\infty$ for every $h>0$) and
$\int_{0}^{t}P\left(  X>s\right)  ds /EX^{+}$ is subexponential (as a function
of $t$) then
\begin{equation}
u(b)=-(\mu^{-1}+o(1))G(b), \label{App}%
\end{equation}
as $b\rightarrow\infty$.
\end{theorem}

We close this subsection with a series of lemmas involving several properties
which will be useful throughout the paper. The proofs of these results are given
in Appendix \ref{SecTech}.


\bigskip

\begin{lemma}
\label{LemRateFun}If B2 holds then $\lambda\left(  x\right)  =O\left(
x^{\beta_{0}-1}\right)  \rightarrow0$ as $x\rightarrow\infty$.
\end{lemma}

\begin{lemma}
\label{LemIntTail} Under Assumption B3 there exists a constant $\kappa_{1}$
(depending on $a_{\ast}$) and $b_{0}$, such that for all $x\leq b-\Lambda
^{-1}(\Lambda(b)-a_{\ast})$ and $b>b_{0}$, the integrated tail satisfies
\[
G(b-x)/G(b)\leq\kappa_{1}.
\]

\end{lemma}

\begin{lemma}
\label{LemIntTail1} Suppose B1 and B3 are in force. For each $\varepsilon
_{0}>0$, there exists $b_{0}>0$ such that
\[
\varepsilon_{0}^{-1}\bar{F}(b)\leq G(b)\leq\varepsilon_{0}b\bar{F}(b),
\]
for all $b\geq b_{0}$. In particular, $\bar{F}(b)/G\left(  b\right)
=o\left(  1\right)  $ as $b\longrightarrow\infty$. If Assumption A holds then for each $\delta_{0}>0$ we can select $b_0>0$
sufficiently large so that
\[
\frac{1-\delta_{0}}{\iota -1}b\bar{F}(b)\leq G(b)\leq \frac{1+\delta_{0}}{\iota -1}b\bar{F}(b).
\]
for $b\geq b_{0}$, where $\iota$ is the tail index of $\bar F$ defined in
Assumption A.
\end{lemma}

\begin{lemma}
\label{LemCen} Suppose B2 holds, for all $x\geq b_{0}$ and $y\geq0$ we have%
\[
\frac{\Lambda(x)}{\Lambda(x+y)}\geq\left(  \frac{x}{x+y}\right)  ^{\beta_{0}%
}.
\]

\end{lemma}

\begin{lemma}
\label{LemOvershoot} Suppose B2 is satisfied. Then, we can choose $b_{0}>0$
sufficiently large such that
\[
x-\Lambda^{-1}(\Lambda(x)-a_{\ast})\geq x^{(1-\beta_{0})/2},
\]
for all $x>b_{0}$.
\end{lemma}

The following lemma allows us to conclude that the Pakes-Veraberbeke theorem
is applicable in our setting.

\begin{lemma}
\label{LemSubExp}Under either Assumption A or B1-3, both $F(x)$ and $\int
_{0}^{x }P\left(  X>s\right)  ds /\left(  EX^{+}\right)  $ are subexponential
as a function of $x$.
\end{lemma}

\subsection{State-dependent importance sampling for the first passage time
random walk problem and Lyapunov inequalities}

Consider two probability measures $P$ and $Q$ on a given space $\mathcal{X}$
with $\sigma$-algebra $\mathcal{F}$. If the Radon-Nikodym derivative
$\frac{dP}{dQ}(\omega)$ is well defined on the set $A\in\mathcal{F}$, then
\[
P(A)=\int\frac{dP}{dQ}(\omega)I_{A}\left(  \omega\right)  Q(d\omega).
\]
We say that the random variable $\frac{dP}{dQ}(\omega)I_{A}\left(
\omega\right)  $ is the importance sampling estimator associated to the
change of measure / importance sampling distribution $Q$. If one chooses
$Q^{\prime}$ such that for each $B\in\mathcal{F}$,
\[
Q^{\prime}(B)=P(B\cap A)/P(A),
\]
then, $\frac{dP}{dQ^{\prime}}\equiv P(A)$ almost surely on the set $A$ and
therefore the estimator $\frac{dP}{dQ^{\prime}}\left(  \omega\right)  $ has
zero variance. This implies that the best importance sampling distribution
(with zero variance for estimating $P(A)$) is the conditional distribution
given the event $A$ occurs.

Certainly, this zero variance estimator is not implementable in practice,
because the Radon-Nikodym derivative involves precisely computing $P(A)$,
which is the quantity to compute. Nevertheless, it provides a general
guideline on how to construct efficient importance sampling estimators: try to
mimic the conditional distribution given the event of interest.

In the context of this paper, we consider a random walk $(S_{n}:n\geq0)$ with
$S_{0}=0$ and therefore%
\[
P(X_{n+1}\in dx|S_{1},...,S_{n})=F(dx).
\]
A state-dependent importance sampling distribution $Q$ is such that
\begin{equation}
Q(X_{n+1}\in dx|S_{1},...,S_{n})=r_{S_{n}}^{-1}(x)F(dx), \label{MarkovCoM}%
\end{equation}
where, the function $(r_{s}\left(  x\right)  :s,x\in R)$ is non-negative and
it satisfies
\[
\int_{-\infty}^{\infty}r_{S_{n}}^{-1}(x)F(dx)=1.
\]

Now, consider the stopping time $\tau_{b}=\inf\{n\geq0:S_{n}>b\}$ and set
$A_{b}=\{\tau_{b}<\infty\}$, then it follows easily that%
\[
P(A_{b})=E^{Q}\left\{  I_{A_{b}}\prod_{i=1}^{\tau_{b}}r_{S_{i-1}}%
(S_{i}-S_{i-1})\right\}  .
\]
\textbf{Notational convention: }throughout the paper we shall use $E_{s}%
^{Q}\left(  \cdot\right)  $ to denote the expectation operator induced by
(\ref{MarkovCoM}) assuming that $S_{0}=s$. We simply write $E^{Q}\left(
\cdot\right)  $ whenever $S_{0}=0$.

\bigskip

We will work with the specific parametric selection of $r_{s}(x)$ introduced
in Section 2. In proving some of our main results we will be interested in
finding an upper bound for the second moment of our estimator under
$E^{Q}\left(  \cdot\right)  $, namely%
\[
E^{Q}\left\{  I_{A}\prod_{i=1}^{\tau_{b}}r_{S_{i-1}}^{2}(S_{i}-S_{i-1}%
)\right\}  =E\left\{  I_{A}\prod_{i=1}^{\tau_{b}}r_{S_{i-1}}(S_{i}%
-S_{i-1})\right\}  .
\]
In general, the $(1+\gamma)$-th moment ($\gamma>0$) of our estimator satisfies%
\[
E\left\{  I_{A}\prod_{i=1}^{\tau_{b}}r_{S_{i-1}}(S_{i}-S_{i-1})^{\gamma
}\right\}
\]
The next lemma provides the mechanism that we shall use to obtain upper bounds
for these quantities. The proof can be found in \cite{BlaGly07}.

\begin{lemma}
\label{LemLyap} Assume that there exists a non-negative function
$g:\mathbb{R}\rightarrow\mathbb{R}^{+}$, such that for all $s<b$,
\[
g(s)\geq E(g(s+X)r_{s}(X)^{\gamma}),
\]
where $X$ is a random variable with density $f(\cdot)$ and suppose that for
all $s\geq b$, $g(s)\geq\varepsilon$. Then,
\begin{equation}
g(0)\geq\varepsilon E\left\{  I_{A}\prod_{i=1}^{\tau_{b}}r_{S_{i-1}}%
(S_{i}-S_{i-1})^{\gamma}\right\}  . \label{LypIneq}%
\end{equation}

\end{lemma}

\bigskip

Most of the time we will work with $\gamma=1$ (i.e. we concentrate on the
second moment). The inequality (\ref{LypIneq}) is said to be a Lyapunov
inequality. The function $g$ is called a Lyapunov function. Lemma
\ref{LemLyap} provides a handy tool to derive an upper bound of the second
moment of the importance sampling estimator. However, the lemma does not
provide a recipe on how to construct a suitable Lyapunov function. We will
discuss the intuition behind the construction of our Lyapunov function in
future sections.

\bigskip

If $r_{s}(x)$ has been chosen in such a way that the second moment of the
importance sampling estimator can be suitably controlled by an appropriate
selection of a Lyapunov function $g$, we still need to make sure that the cost
per replication (i.e. $E^{Q}\tau_{b}$) is suitably controlled as well. The
next lemma, which follows exactly the same steps as in the first part of the
proof in Theorem 11.3.4 of \cite{MT93}, establishes a Lyapunov criterion
required to control the behavior of $E^{Q}\tau_{b}$.

\begin{lemma}
\label{LemTM}Suppose that one can find a non-negative function $h(\cdot)$ and
a constant $\rho>0$ so that
\[
E_{s}^{Q}(h(s+X))\leq h(s)-\rho,
\]
for $s<b$. Then, $E^{Q}(\tau_{b}|S_{0}=s)\leq h(s)/\rho$ for $s<b$.
\end{lemma}

\bigskip

Most of the results discussed in Section 2 of the paper involve constructing
suitable selections of Lyapunov functions $g$ and $h$ appearing in the
previous lemmas. The construction of these functions is given in subsequent sections.

\section{Lyapunov function for variance control}

\label{SectionAnalysis}

Our approach to designing efficient importance sampling estimators consist of
three steps:

\begin{enumerate}
\item Propose a family of change of measures suitably parameterized.

\item Propose candidates of  Lyapunov functions using fluid heuristics and also depending on
appropriate parameters.

\item Verify the Lyapunov inequality by choosing appropriate parameters for
the change of measure and the Lyapunov function.
\end{enumerate}

Our family has been introduced in Section 2. This corresponds to the first
step. The second and third steps are done simultaneously. We will choose the
parameters $\eta_{*}$, the $c_{j}$'s, $a_{\ast}$, $a_{\ast\ast}$, $p_{\ast}$,
$p_{\ast\ast}$ and the $p_{j}$'s of our change of measure in order to satisfy
an appropriate Lyapunov function for variance control by means of Lemma
\ref{LemLyap}. Some of the parameters, in particular the $c_{j}$'s, can be set
in advance without resorting to the appropriate Lyapunov function. The key
element is given in the next lemma, whose proof is given in the appendix.

\begin{lemma}
\label{LemWeibull} Fix $\beta_{0}\in(0,1)$ and select $\sigma_{1}>0$ sufficiently small such that for every $x\in [0,\sigma_1]$ $2- 2(1-x)^{\beta_0} - x^{\beta_0}\leq 0$. Then, there
exists $\sigma_{2}>0$ and a sequence, $0<a_{1}<a_{2}<\cdots<a_{k-1}<1$ such
that $a_{j+1}-a_{j}\leq\sigma_{1}/2$ for each $1\leq j\leq k-2$,
\[
a_{j}^{\beta_{0}}+(1-a_{j+1})^{\beta_{0}}\geq1+\sigma_{2}.
\]
and $a_{k-1}\geq1-\sigma_{1}$, $a_{1}\leq\sigma_{1}$.
\end{lemma}

Given $\beta_{0}$ in Assumption B2, from now on, we choose
\begin{equation}
c_{0} = b-s - \Lambda^{-1} (\Lambda(b-s)-a_{*}),\quad c_{k} = \Lambda^{-1}
(\Lambda(b-s)-a_{**}),\quad c_{j}=a_{j}(b-s), \label{c}%
\end{equation}
for $j=1,...,k-1$, with $\sigma_{1}$ chosen small enough and $a_{j}%
=a_{j-1}+\sigma_{1}/2$ according to the previous lemma.

We continue with the second step of our program. We concentrate on bounding
the second moment and discuss the case of ($1+\gamma)$-th moment later. The
value of the Lyapunov function at the origin, namely, $g\left(  0\right)  $ in
Lemma \ref{LemLyap} serves as the upper bound of the second moment of the
importance sampling estimator. In order to prove strong efficiency, we aim to
show that there exists a constant $c<\infty$ such that
\[
E^{Q}Z_{b}^{2}\leq cu^{2}(b),
\]
where
\begin{equation}
Z_{b}=I(\tau_{b}<\infty)\prod_{i=1}^{\tau_{b}}r_{S_{i-1}}(S_{i}-S_{i-1})
\label{LR}%
\end{equation}
is the estimator of $u(b)$. Therefore, a useful Lyapunov function for proving
strong efficiency must satisfy that
\[
g(0)\leq cu^{2}(b).
\]
It is natural to consider using an approximation of $u^{2}(b-s)$ as the
candidate. Exactly the same type of fluid heuristic analysis that we used in
(\ref{fluid}) suggests
\begin{equation}
g(s)=\min\{\kappa G^{2}(b-s),1\}, \label{Lyp}%
\end{equation}
where $G$ is the integrated tail defined in (\ref{IntTail}) and $\kappa$ is a
non-negative tuning parameter which will be determined later.

It is important to keep in mind that $g(s)$ certainly depends on $b$. For
notational simplicity, we omit the parameter $b$. The function $g\left(
s\right)  $ will also dictate when we are close enough to the boundary level
$b$ where importance sampling is not required. In particular, using our
notation in (\ref{Sel_q}) and (\ref{MarkovCoM}) we propose choosing $\eta_{\ast}=G^{-1}\left(
\kappa^{-1/2}\right)  $ which amounts to choosing%
\begin{align*}
&  r_{s}\left(  x\right)  ^{-1}\\
&  =\left(  \frac{p_{\ast}I(x\leq c_{0})}{P(X\leq c_{0})}+\frac{p_{\ast\ast
}I(x>c_{k})}{P(X>c_{k})}+\sum_{j=1}^{k-1}\frac{p_{j}I(x\in(c_{j-1},c_{j}%
])}{P(X\in(c_{j-1},c_{j}])}+\frac{f(b-s-x)p_{k}I(x\in(c_{k-1},c_{k}%
])}{f(x)P(X\in(b-s-c_{k},b-s-c_{k-1}])}\right) \\
&  \times I\left(  g\left(  s\right)  <1\right)  +I\left(  g\left(  s\right)
=1\right)  .
\end{align*}

Now we proceed to the last step -- the verification of the Lyapunov
inequality. The Lyapunov inequality in Lemma \ref{LemLyap} is equivalent to
\begin{equation}
\frac{E(r_{s}(X)g(s+X))}{g(s)}\leq1. \label{LypSep}%
\end{equation}
The interesting part of the analysis is the case $g\left(  s\right)  <1$
because whenever $g\left(  s\right)  =1$ the inequality is trivially satisfied
given that $0\leq g\left(  s+X\right)  \leq1$. Hereafter, we will focus on the
case that $g\left(  s\right)  <1$.

The left hand side of (\ref{LypSep}) can be decomposed into the following
pieces,
\begin{align*}
\frac{E(r_{s}(X)g(s+X))}{g(s)}  &  =\frac{P(X\leq b-s-\Lambda^{-1}%
(\Lambda(b-s)-a_{\ast}))}{p_{\ast}}\\
&  \times E\left(  \frac{g(s+X)}{g(s)};X\leq b-s-\Lambda^{-1}(\Lambda
(b-s)-a_{\ast})\right) \\
&  +\frac{P(X>\Lambda^{-1}(\Lambda(b-s)-a_{\ast\ast}))}{p_{\ast\ast}}E\left(
\frac{g(s+X)}{g(s)};X>\Lambda^{-1}(\Lambda(b-s)-a_{\ast\ast})\right) \\
&  +\sum_{i=1}^{k-1}\frac{P(X\in(c_{i-1},c_{i}])}{p_{i}}E\left(  \frac
{g(s+X)}{g(s)};X\in(c_{i-1},c_{i}]\right) \\
&  +\frac{P(b-s-X\in(c_{k-1},c_{k}])}{p_{k}}E\left(  \frac{g(s+X)f(X)}%
{g(s)f(b-s-X)};X\in(c_{k-1},c_{k}]\right)  .
\end{align*}
We adopt the following notation
\begin{align}
J_{\ast}  &  =P(X\leq b-s-\Lambda^{-1}(\Lambda(b-s)-a_{\ast}))E\left(
\frac{g(s+X)}{g(s)};X\leq b-s-\Lambda^{-1}(\Lambda(b-s)-a_{\ast})\right)
\label{J*}\\
J_{\ast\ast}  &  =P(X>\Lambda^{-1}(\Lambda(b-s)-a_{\ast\ast}))E\left(
\frac{g(s+X)}{g(s)};X>\Lambda^{-1}(\Lambda(b-s)-a_{\ast\ast})\right)
\label{J**}\\
J_{i}  &  =P(X\in(c_{i-1},c_{i}])E\left(  \frac{g(s+X)}{g(s)};X\in
(c_{i-1},c_{i}]\right)  \text{, for }i=1,...,k-1\label{Ji}\\
J_{k}  &  =P(b-s-X\in(c_{k-1},c_{k}])E\left(  \frac{g(s+X)f(X)}{g(s)f(b-s-X)}%
;X\in(c_{k-1},c_{k}]\right)  , \label{Jk}%
\end{align}
so that inequality (\ref{LypSep}) is equivalent to showing that%
\[
\frac{J_{\ast}}{p_{\ast}}+\frac{J_{\ast\ast}}{p_{\ast\ast}}+\sum_{i=1}%
^{k-1}\frac{J_{i}}{p_{i}}+\frac{J_{k}}{p_{k}}\leq 1.
\]
We shall study each of these terms separately.

\bigskip

At this point it is useful to provide a summary of all the relevant constants
and parameters introduced so far:

\begin{itemize}
\item $\iota>1$ is the regularly varying index under Assumption A.

\item $b_{0}>0$ is introduced in Assumption B, Lemmas \ref{LemIntTail1} and \ref{LemOvershoot} to ensure regularity properties.


\item $\beta_{0}\in(0,1)$ is introduced in B2 to guarantee that the distribution
considered is ``heavier'' than a Weibull distribution with shape parameter
$\beta_{0}$

\item $a_{\ast}$, $a_{\ast\ast}>0$ are introduced to define the mixture components
corresponding to a ``regular jump" and a ``large jump" respectively.

\item $a_{1}<...<a_{k-1}$ are defined according to Lemma \ref{LemWeibull}.

\item $c_{j}$ for $j=0,1,...,k$ are defined in \eqref{c} and correspond to the
end points of the support of the interpolating mixture components.

\item $\kappa,$ $\eta_{\ast}$ are parameters for the Lyapunov function. They
are basically equivalent since $\eta_{\ast}=G^{-1}\left(  \kappa
^{-1/2}\right)  $, $\kappa$ appears in the definition of the Lyapunov
function. It is important to keep in mind that by letting $\kappa$ be large,
the condition $g\left(  s\right)  <1$ implies that $b-s>\eta_{\ast}$ is large.

\item $\varepsilon_{0},\delta_{0}$ are arbitrarily small constants introduced
in Lemma \ref{LemIntTail1}.


\item The parameters $p_{\ast}$, $p_{\ast\ast}$ and $p_{i}$ for $i=1,...,k$ are the mixture probabilities and
will depend on the current state $s$.
\end{itemize}

Other critical constants which will be introduced in the sequel concerning the
analysis of $J_{\ast}$, $J_{\ast\ast}$, and $J_{i}$, $i=1,...,k$ are:

\begin{itemize}
\item $\delta_{0}^{\ast}>0$ is a small parameter which appears in the analysis
of $J_{\ast}$. It will be introduced in Proposition \ref{PropJReg}.

\item $\delta_{1}^{\ast}>0$, a small parameter, appears in the definition of
$p_{i}$ and the overall contribution of the $J_{i}$'s. It will be introduced
in step III) of the parameter selection process.

\item $\delta_{2}^{\ast}>0$ is introduced to control the termination time of
the algorithm. It ultimately provides a link between $a_{\ast\ast}>0$ and
$\delta_{0}^{\ast}>0$ in Section \ref{SecTermination}.


\item Parameters $\theta$, $\tilde{\varepsilon}$ and $\tilde{\varepsilon}_{1}$
which are introduced to specify the probabilities $p_{\ast\ast}$ and the
$p_{i}$'s respectively. Their specific values depending on $\delta_{0}^{\ast}$
and $\delta_{1}^{\ast}$ will be indicated in steps I) to IV) below.
\end{itemize}

\bigskip

Throughout the rest of the paper we shall use $\varepsilon,\delta>0$ to denote
arbitrarily small positive constants whose values might even change from line
to line. Similarly, $K,c\in\left(  0,\infty\right)  $ are used to denote
positive constants that will be employed as generic upper bounds.

\bigskip

Now, we study the terms $J_{\ast}$, $J_{\ast\ast}$, and $J_{i}$, $i=1,...,k$.

\paragraph{The term}

$J_{\ast\ast}$:
\begin{align}
J_{\ast\ast}  &  =P(X>\Lambda^{-1}(\Lambda(b-s)-a_{\ast\ast}))E\left(
\frac{g(s+X)}{g(s)};X>\Lambda^{-1}(\Lambda(b-s)-a_{\ast\ast})\right)
\nonumber\\
&  \leq\frac{P^{2}(X>\Lambda^{-1}(\Lambda(b-s)-a_{\ast\ast}))}{g(s)}%
=e^{2a_{\ast\ast}}\frac{\bar{F}^{2}(b-s)}{g(s)} \label{JTail}%
\end{align}

\paragraph{A bound for}

$J_{\ast}$:

\begin{proposition}
\label{PropJReg} Suppose the distribution function $F$ satisfies Assumption A
or Assumptions B1-3. Then, as $b-s\rightarrow\infty$,
\[
E\left(  \frac{g(s+X)}{g(s)};X\leq b-s-\Lambda^{-1}(\Lambda(b-s)-a_{\ast
})\right)  \leq1+(1+o(1))\mu\frac{\partial g(s)}{g\left(  s\right)  }.
\]
Therefore, for any $\delta^*_0>0$, we can select $\eta_{*}>0$ such that for
all $b-s>\eta_{*}$,
\[
E\left(  \frac{g(s+X)}{g(s)};X\leq b-s-\Lambda^{-1}(\Lambda(b-s)-a_{\ast
})\right)  \leq1+\mu(1-\delta^{\ast}_{0})\frac{\partial g(s)}{g\left(
s\right)  }.
\]

\end{proposition}

\begin{proof}
[Proof of Proposition \ref{PropJReg}]By Taylor's expansion,
\[
\frac{g(s+X)}{g(s)}=1+X\frac{\partial g(s+\xi)}{g(s)},
\]
where $\xi\in(0,X)$ (or $(X,0)$). For all $s$ and $X$ such that $g(s)<1$ and
$g(s+X)<1$,
\[
X\partial g(s+\xi)/g(s)=2X\bar{F}(b-s-\xi)G(b-s-\xi)/G^{2}(b-s)=2X\frac
{\bar{F}(b-s-\xi)}{\bar{F}(b-s)}\frac{G(b-s-\xi)}{G(b-s)}\frac{\bar{F}%
(b-s)}{G(b-s)}.
\]
Then,
\begin{align*}
&  \frac{G(b-s)}{\bar{F}(b-s)}E\left(  X\partial g(s+\xi)/g(s);X\leq
b-s-\Lambda^{-1}(\Lambda(b-s)-a_{\ast})\right) \\
&  \leq2E\left(  X\frac{\bar{F}(b-s-\xi)}{\bar{F}(b-s)}\frac{G(b-s-\xi
)}{G(b-s)};X\leq b-s-\Lambda^{-1}(\Lambda(b-s)-a_{\ast})\right)
\end{align*}
Note the following facts,
\[
\frac{\bar{F}(b-s-\xi)}{\bar{F}(b-s)}\leq e^{a_{\ast}},
\]
and by Lemma \ref{LemIntTail} (Assumption B) or the regularly variation property of $G$ (Assumption A),
\[
\frac{G(b-s-\xi)}{G(b-s)}\leq\kappa_{1},
\]
and by Lemma \ref{LemIntTail1} and the fact that $F$ is subexponential (Lemma
\ref{LemSubExp}),
\[
X\frac{\bar{F}(b-s-\xi)}{\bar{F}(b-s)}\frac{G(b-s-\xi)}{G(b-s)}\rightarrow X,
\]
as $b-s\rightarrow\infty$.
By the dominated convergence theorem,
\begin{equation}
\lim_{b-s\rightarrow\infty}\frac{G(b-s)}{\bar{F}(b-s)}E\left(
X\partial g(s+\xi)/g(s);X\leq b-s-\Lambda^{-1}(\Lambda(b-s)-a_{\ast})\right)
=2\mu. \label{DCGT}%
\end{equation}
Therefore, we can always choose the constants appropriately such that the
conclusion of the proposition holds.
\end{proof}

\bigskip

As remarked in equation (\ref{k0}), the terms $J_{i},\ i=1,...,k$, do not
appear in the context of Assumption A. We consider them in the context of
Assumption B.

\paragraph{Bound for $J_{i}$, $2\leq i\leq k-1$:}

\begin{proposition}
\label{PropMid} Suppose that Assumptions B1-3 hold. Then, for each $2\leq
i\leq k-1$, we have that for any $\alpha>0$
\[
J_{i}=\int_{c_{i-1}}^{c_{i}}\frac{f(x)g(s+x)}{f_{j}(x)g\left(  s\right)
}f(x)dx=o(\left(  b-s\right)  ^{-\alpha}),
\]
as $b-s\rightarrow\infty$.
\end{proposition}

\begin{proof}
Thanks to Lemma \ref{LemCen}, for each $x,y,z$ sufficiently large, we have
\begin{equation}
\Lambda(x)+\Lambda(y)-\Lambda(x+y+z)\geq\Lambda(x+y+z)\left(  \left(  \frac
{x}{x+y+z}\right)  ^{\beta_{0}}+\left(  \frac{y}{x+y+z}\right)  ^{\beta_{0}%
}-1\right)  . \label{IneqHaz}%
\end{equation}
We first note that by repeatedly using results in Lemma \ref{LemIntTail1}
\begin{align*}
&  \int_{c_{j-1}}^{c_{j}}\frac{f(x)g(s+x)}{\kappa f_{j}(x)G^{2}(b-s)}f(x)dx\\
&  =\frac{P(X\in(c_{j-1},c_{j}])}{G^{2}(b-s)}\int_{c_{j-1}}^{c_{j}}%
G^{2}(b-s-x)f(x)dx\\
&  \leq\frac{\varepsilon_{0}e^{\Lambda(b-s)-\Lambda(c_{j-1})}}{ G(b-s)}%
\int_{c_{j-1}}^{c_{j}}G^{2}(b-s-x)\lambda(x)e^{-\Lambda(x)}dx\\
&  \leq\frac{\varepsilon_{0}e^{\Lambda(b-s)-\Lambda(c_{j-1})}}{G(b-s)}\bar
{F}(c_{j-1})G^{2}(b-s-c_{j})\\
&  \leq\varepsilon_{0}^{4}\frac{e^{\Lambda(b-s)-\Lambda(c_{j-1})}}{\bar
{F}(b-s)}\bar{F}(c_{j-1})(b-s)^{2}\bar{F}^{2}(b-s-c_{j})\\
&  =\varepsilon_{0}^{4}(b-s)^{2}e^{2\Lambda(b-s)-2\Lambda(c_{j-1}%
)-2\Lambda(b-s-c_{j})}\\
&  \leq\varepsilon_{0}^{4}(b-s)^{2}\exp\left\{  -2\Lambda(b-s)\left(
a_{j-1}^{\beta_{0}}+(1-a_{j})^{\beta_{0}}-1\right)  \right\}  =o(1) \left(
b-s\right)  ^{-\alpha},
\end{align*}
as $b-s\rightarrow\infty$ for each $\alpha>0$. The last inequality is thanks to
(\ref{c}), (\ref{IneqHaz}). The last step (equality) follows from Lemma
\ref{LemWeibull} and Assumption B1 which implies that the tail of $X$
decreases faster than any polynomial.
\end{proof}

\paragraph{A bound for $J_{1}$:}

\begin{proposition}
\label{PropFirst} Suppose that Assumptions B1-3 hold. Then, for each
$\alpha>0$ we have%
\[
J_{1}=\int_{b-s-\Lambda^{-1}(\Lambda(b-s)-a_{\ast})}^{c_{1}}\frac
{f(x)g(s+x)}{f_{1}(x)g\left(  s\right)  }f(x)dx=o(\left(  b-s\right)
^{-\alpha}),
\]
as $b-s\rightarrow\infty$.
\end{proposition}

\begin{proof}
[Proof of Proposition \ref{PropFirst}]Use Lemma \ref{LemIntTail1} and
$\lim_{x\rightarrow\infty}\lambda(x)=0$ and obtain
\begin{align*}
&  \int_{b-s-\Lambda^{-1}(\Lambda(b-s)-a_{\ast})}^{c_{1}}\frac{f(x)g(s+x)}%
{\kappa f_{1}(x)G^{2}(b-s)}f(x)dx\\
&  \leq\frac{P(X>b-s-\Lambda^{-1}(\Lambda(b-s)-a_{\ast}))}{G^{2}(b-s)}%
\int_{b-s-\Lambda^{-1}(\Lambda(b-s)-a_{\ast})}^{c_{1}}G^{2}(b-s-x)f(x)dx\\
&  \leq\varepsilon_{0}^{4}(b-s)^{2}P(X>b-s-\Lambda^{-1}(\Lambda(b-s)-a_{\ast
}))\int_{b-s-\Lambda^{-1}(\Lambda(b-s)-a_{\ast})}^{c_{1}}e^{2\Lambda
(b-s)-2\Lambda(b-s-x)-\Lambda(x)}dx.
\end{align*}
Also note that by Lemma \ref{LemCen},
\begin{equation}
\Lambda(x)+\Lambda(b-s-x)-\Lambda(b-s)\geq\Lambda(b-s)\left(  \left(  \frac
{x}{b-s}\right)  ^{\beta_{0}}+\left(  \frac{b-s-x}{b-s}\right)  ^{\beta_{0}%
}-1\right)  , \label{ImpIneq}%
\end{equation}
and,
\[
\Lambda(b-s)-\Lambda(b-s-x)\leq\Lambda(b-s)(1-(1-x/(b-s))^{\beta_{0}}).
\]
Therefore, for all $x\in\lbrack b-s-\Lambda^{-1}(\Lambda(b-s)-a_{\ast}%
),\sigma_{1}(b-s)]$, with $\sigma_{1}$ selected according to Lemma \ref{LemWeibull},
\[
2\Lambda(b-s)-2\Lambda(b-s-x)-\Lambda(x)\leq\Lambda(b-s)\left(  2-2\left(
1-\frac{x}{b-s}\right)  ^{\beta_{0}}-\frac{x^{\beta_{0}}}{(b-s)^{\beta_{0}}%
}\right)  \leq0.
\]
Together with Lemma \ref{LemOvershoot}, $P(X>b-s-\Lambda^{-1}(\Lambda
(b-s)-a_{\ast}))$ decreases to zero faster than any polynomial rate. The
conclusion of the lemma follows.
\end{proof}

\paragraph{A bound for $J_{k}$:}

\begin{proposition}
\label{PropLast}If Assumption B holds then for each $\alpha>0$%
\[
J_{k}=\int_{c_{k-1}}^{c_{k}}\frac{f(x)g(s+x)}{f_{k}(x)g\left(  s\right)
}f(x)dx=o(\left(  b-s\right)  ^{-\alpha}),
\]
as $b-s\rightarrow\infty$.
\end{proposition}

\begin{proof}
[Proof of Proposition \ref{PropLast}]Note that
\begin{align*}
&  \int_{c_{k-1}}^{c_{k}}\frac{g(s+x)}{\kappa G^{2}(b-s)}\frac{f^{2}(x)}%
{f_{k}(x)}dx\\
&  =P(X\in(b-s-c_{k},b-s-c_{k-1}])\int_{c_{k-1}}^{c_{k}}\frac{g(s+x)}{\kappa
G^{2}(b-s)}\frac{f^{2}(x)}{f(b-s-x)}dx\\
&  \leq\varepsilon_{0}^{4}\bar{F}(b-s-c_{k})\int_{c_{k-1}}^{c_{k}}%
\frac{(b-s)^{2}\lambda^{2}(x)}{\lambda(b-s-x)}e^{2\Lambda(b-s)-2\Lambda
(x)-\Lambda(b-s-x)}dx.
\end{align*}
We note that $\sigma_{1}$ is small enough and $x>(1-\sigma_{1})(b-s)$ so that
we can apply Lemma \ref{LemWeibull} to conclude%
\[
2\Lambda(b-s)-2\Lambda(x)-\Lambda(b-s-x)\leq \Lambda (b-s) \left (2-2\left(  \frac{x}{b-s}\right)
^{\beta_{0}}-\left(  \frac{b-s-x}{b-s}\right)  ^{\beta_{0}} \right )\leq0.
\]
By Assumption B1,
$1/\lambda(x)$ grows at most linearly in $x $ and also we have (just as in Lemma \ref{LemOvershoot}) that $\bar{F}%
(b-s-c_{k})\leq \bar F((b-s)^{\frac{1-\beta_0}{2}})$ decays faster than any polynomial rate. We then have the conclusion of
the proposition.
\end{proof}

\paragraph{Summary of estimates and implications for the design of the
change of measure selection.}

The previous bounds on $J_{\ast}$, $J_{\ast\ast}$, and $J_{i}$, $i=1,...,k$
imply that we can choose parameters and setup the algorithm as follows.

\begin{itemize}
\item[\textbf{I}] If Assumption A holds, we choose $a_*$ and $a_{**}$ such that \eqref{Reg} holds. If Assumption B holds, given $a_{\ast},a_{\ast\ast}>0$, $\sigma_{1}>0$, and $a_{j}=a_{j-1}+\sigma_{1}/2$, chosen according to
Lemma \ref{LemWeibull}, let
\[
c_{0} = b-s - \Lambda^{-1}(\Lambda(b-s)-a_{*}),\quad c_{k} = \Lambda
^{-1}(\Lambda(b-s)-a_{**}),
\]
$c_{j}=a_{j}(b-s)$ for $j=1,...,k-1$.

\item[\textbf{II}] Select $\delta_{0}^{\ast}\in(0,1/4)$ and let $\eta_{*}>0$
be large enough so that if $b-s>\eta_{*}$ then
\begin{equation}
\frac{J_{\ast}}{p_{\ast}}\leq\frac{1}{p_{\ast}}+\frac{(1-\delta_{0}^{\ast}%
)}{p_{\ast}}\mu\frac{\partial g(s)}{g(s)}. \label{SBJ*}%
\end{equation}

\item[\textbf{III}] Choose $\delta_{1}^{\ast}\in(0,\delta_0^*\mu^{2}(1-\delta_{0}^{\ast
})^{2}(1+\delta_{0}^{\ast})^{-10}/(k+1)^{2})$ such that if $b-s>\eta_{*}$ for
$\eta_{*}$ large enough
\[
J_{i}\leq\delta_{1}^{\ast}\delta_{0}^{\ast}\left(  \frac{\partial g(s)}%
{g(s)}\right)  ^{2}%
\]
for all $i=1,...,k$. Note that the $J_i$ terms are all zero for the regularly varying case.
\end{itemize}

The choice in III) is feasible because $\partial g\left(  s\right)  /g\left(
s\right)  =2\bar{F}(b-s)/G\left(  b-s\right)  $ decreases at most a polynomial
rate and $J_{i}$ terms derived in Propositions \ref{PropMid}, \ref{PropFirst},
and \ref{PropLast} are smaller than any polynomial rate. Both II) and III) can
be satisfied simultaneously by choosing $\eta_{*}$ sufficiently large. Now,
with the selections in II) and III) we have that%
\begin{align}
&  \frac{J_{\ast}}{p_{\ast}}+\frac{J_{\ast\ast}}{p_{\ast\ast}}+\sum_{i=1}%
^{k}\frac{J_{k}}{p_{k}}\nonumber\\
&  \leq\frac{1}{p_{\ast}}+\frac{(1-\delta_{0}^{\ast})}{p_{\ast}}\mu
\frac{\partial g(s)}{g(s)}+e^{2a_{\ast\ast}}\frac{\bar{F}^{2}(b-s)}%
{p_{\ast\ast}g(s)}+\delta_{1}^{\ast}\delta_{0}^{\ast}\left(  \frac{\partial
g(s)}{g(s)}\right)  ^{2}\sum_{i=1}^{k}\frac{1}{p_{i}}. \label{Sum}%
\end{align}
Now we must select $p_{\ast}$, $p_{\ast\ast}$ and the $p_{i}$'s so that
(\ref{Sum}) is less than unity in order to satisfy (\ref{LypSep}). Recall that
$p_{\ast\ast}$ represents the mixture probability associated to the occurrence
of the rare event in the next step. Therefore, it makes sense to select
$p_{\ast\ast}$ of order $\Theta(\bar{F}(b-s)/G\left(  b-s\right)  )$ as
$b-s\rightarrow\infty$. Motivated by this observation and given the analytical
form of the equation above we write
\begin{equation}
p_{\ast\ast}=\min\{\theta\partial g(s)/g(s),\widetilde{\varepsilon}%
\}=\min\{2\theta\bar{F}(b-s)/G(b-s),\widetilde{\varepsilon}\} \label{Selp**}%
\end{equation}
for some $\theta,\widetilde{\varepsilon}>0\ $(the precise values of $\theta$
and $\widetilde{\varepsilon}$ will be given momentarily) and let
\begin{equation}
p_{i}=\widetilde{\varepsilon}_{1}p_{\ast\ast} \label{Selpi}%
\end{equation}
for each $i=1,...,k$ for some $\widetilde{\varepsilon}_{1}>0$ small enough to
be defined shortly. This selection of $p_{i}$'s also makes intuitive sense
because the corresponding mixture terms will give rise to increments that are
large, yet not large enough to reach the level $b$ of the random walk and
therefore they correspond to ``rogue paths" -- as we called them in the
Introduction. In addition, one can always choose $\eta_*$ large enough such that $p_{**}<\tilde \varepsilon$ for all $b-s >\eta_*$. Given these selections we obtain
\begin{equation}
p_{\ast}=1-p_{\ast\ast}-k\widetilde{\varepsilon}_{1}p_{\ast\ast}.
\label{Selp*}%
\end{equation}

We then conclude that if $p_{\ast\ast}(1+k\widetilde{\varepsilon}_{1}%
)<\delta_{0}^{\ast}/2<1/4$ and $\tilde\varepsilon< \delta_{0}^{*}/2$, then%
\begin{align*}
&  \frac{J_{\ast}}{p_{\ast}}+\frac{J_{\ast\ast}}{p_{\ast\ast}}+\sum_{i=1}%
^{k}\frac{J_{k}}{p_{k}}\\
&  \leq1+p_{\ast\ast}\left(  1+k\widetilde{\varepsilon}_{1}\right)  \left(
1-\delta_{0}^{\ast}\right)  ^{-1} +\frac{(1-\delta_{0}^{\ast})}{\theta}\mu
p_{\ast\ast}+e^{2a_{\ast\ast}}\frac{p_{\ast\ast}}{4\theta^{2}\kappa}%
+k\delta_{1}^{\ast}\delta_{0}^{\ast}\frac{p_{\ast\ast}}{\theta^{2}%
\widetilde{\varepsilon}_{1}}\\
&  =1+p_{\ast\ast}\left[  \left(  1+k\widetilde{\varepsilon}_{1}\right)
\left(  1-\delta_{0}^{\ast}\right)  ^{-1} +\frac{(1-\delta_{0}^{\ast})}%
{\theta}\mu+\frac{e^{2a_{\ast\ast}}}{4\theta^{2}\kappa}+k\frac{\delta
_{1}^{\ast}\delta_{0}^{\ast}}{\theta^{2}\widetilde{\varepsilon}_{1}}\right]
.
\end{align*}
Now choose $\widetilde{\varepsilon}_{1}=\delta_{0}^{\ast}/(k+1)$ and then
select $\theta=-\mu(1-\delta_{0}^{\ast})/(1+\delta_{0}^{\ast})^{5}$. Then we
note that our selection of $\delta_{1}^{\ast}$ guarantees $\delta_{1}^{\ast
}\leq\theta^{2}\widetilde{\varepsilon}_{1}k^{-1}$. Finally it is required that
$\kappa\geq e^{2a_{\ast\ast}}/[4\theta^{2}\delta^{\ast}_{0}]$. Note that the
selection of $\delta_{0}^{\ast},\delta_{1}^{\ast}>0$ requires that
$b-s>\eta_{*}$ for $\eta_{*}>0$ sufficiently large, which is guaranteed
whenever $g\left(  s\right)  <1$ and $\kappa$ is sufficiently large. So, the
selection of $\kappa$ might possibly need to be increased in order to satisfy
all the constraints. All this selections in place yield (using the fact that
$\delta_{0}^{\ast}<1/4$)%
\[
\frac{J_{\ast}}{p_{\ast}}+\frac{J_{\ast\ast}}{p_{\ast\ast}}+\sum_{i=1}%
^{k}\frac{J_{k}}{p_{k}}\leq1+p_{\ast\ast}\left(  \left(  1+\delta_{0}^{\ast
}\right)  ^{2}-(1+\delta_{0}^{\ast})^{5}+2\delta_{0}^{\ast}\right)
\leq1+p_{\ast\ast}\delta_{0}^{\ast}\left(  \delta_{0}^{\ast}-1\right)  \leq1.
\]
The various parameter selections based on the previous discussion are
summarized next.

\begin{itemize}
\item[\textbf{IV}] Select $\widetilde{\varepsilon}_{1}=\delta_{0}^{\ast
}/(k+1)$, $\widetilde{\varepsilon}=(\delta_{0}^{\ast})^{2}$ (this guarantees
$p_{\ast\ast}(1+k\widetilde{\varepsilon}_{1})<\delta_{0}^{\ast}/2$) and
$\theta=-\mu(1-\delta_{0}^{\ast})/(1+\delta_{0}^{\ast})^{5}$. Set $p_{\ast
\ast}$, $p_{i}$ for $i=1,...,k$ and $p_{\ast}$ according to (\ref{Selp**}),
(\ref{Selpi}) and (\ref{Selp*}) respectively. Then, choose $\kappa$ large
enough so that $\kappa\geq e^{2a_{\ast\ast}}/[4\theta^{2}\delta_{0}^{\ast}]$
and at the same time $g\left(  s\right)  <1$ implies $b-s>\eta_{\ast}$, with
$\eta_{\ast}$ also appearing in II) above.
\end{itemize}

We now can provide a precise description of the importance sampling scheme.
Assume that the selection procedure indicated from I) to IV)\ above has been
performed and let $S_{0}=0$. Suppose that the current position at time $k$,
namely $S_{k}$, is equal to $s$ and that $\tau_{b}>k$. We simulate the
increment $X_{k+1}$ according to the following law. If $g\left(  s\right)  <1$
then we sample $X_{k+1}$ with the mixture density in \eqref{SampDist}.
Otherwise, if $g\left(  s\right)  =1$ we sample $X_{k+1}$ with density
$f\left(  \cdot\right)  $. The corresponding importance sampling estimator is
precisely%
\begin{equation}
Z_b=I(\tau_{b}<\infty)\prod_{i=1}^{\tau_{b}}r_{S_{i-1}}(S_{i}-S_{i-1}).
\label{EstZThm}%
\end{equation}
Note that we have not discussed the termination of the algorithm -- the
expected value of $\tau_{b}$ under the proposed importance sampling distribution.
Indeed, this is an issue that will be studied in the next section.
Here we are only interested in the variance analysis of $Z_b$.

\bigskip

\begin{proof}
[Proof of Theorem \ref{ThmSE1}]We must show that the estimator $Z_b$ defined in
(\ref{EstZThm}) is strongly efficient for estimating $u(b)$. Our discussion
summarized in the selection process from I) to IV) above indicates that
$g\left(  \cdot\right)  $ is a valid Lyapunov function. Therefore we have
that
\[
E^{Q}Z_b^{2}\leq g(0).
\]
Hence, according to (\ref{App}),
\[
\sup_{b>1}\frac{g(0)}{u^{2}(b)}<\infty.
\]
\end{proof}

\section{Controlling the expected termination time\label{SecTermination}}

As mentioned previously, if $Z_{b}$ is a strongly efficient estimator for
$u(b)$, in order to compute $u(b)$ with $\varepsilon$ relative error with at least
$1-\delta$ probability, one needs to generate $O(\varepsilon^{-2}\delta^{-1})$
(uniformly in $b$) i.i.d. copies of $Z_{b}$. The concept of strong efficiency
by itself does not capture the complexity of generating a single replication
of $Z_{b}$. In this section we will further investigate the computational cost
of generating $Z_{b}$. We shall assume that sampling from the densities
$q_{s}(\cdot)$ or $f\left(  \cdot\right)  $ takes at most a given constant
computational cost, so the analysis reduces to finding a suitable upper bound
for $E^{Q}\tau_{b}$.

We first assume that $F$ is a regularly varying distribution. We will see that
if I) to IV) and also V) below are satisfied then the expected termination time
is $O\left(  b\right)  $. The key message is that we can always select
$a_{\ast\ast},\delta_{0}^{\ast}>0$ sufficiently small in order to satisfy both
Lyapunov functions in Lemmas \ref{LemLyap} and \ref{LemTM}.

\begin{itemize}
\item[\textbf{V}] If Assumption A holds, let $\eta_{*}$ be large enough so
that if $g\left(  s\right)  <1$ (i.e. $b-s>\eta_{*}=G^{-1}\left(
\kappa^{-1/2}\right)  $) then
\[
\frac{\bar{F}(b-s)}{G(b-s)}\geq\frac{\left(  \iota-1\right)  \left(
1-\delta_{0}^{\ast}\right)  }{b-s}.
\]
We also have that $a_{\ast\ast},\delta_{0}^{\ast}>0$ are sufficiently close to
zero such that
\[
\delta_{2}^{\ast}=2(\iota-1)\frac{(1-\delta_{0}^{\ast})^{2}}{\left(
1+\delta_{0}^{\ast}\right)  ^{5}}e^{-a_{\ast\ast}}-1-2(1-e^{-2a_{\ast\ast
}/\iota})\left(  \iota-1\right)  >0
\]
with $\iota> 1.5$.
\end{itemize}

\begin{proposition}
\label{PropReg} Suppose that Assumption A holds and $\iota>1.5$. Then,
the selection indicated in I) to V) yields both Theorem \ref{ThmSE1} and
\[
E^{Q}(\tau_{b})<\rho_{0}+\rho_{1}b,
\]
for $\rho_{0},\rho_{1}\in\left(  0,\infty\right)  $ independent of $b$.
\end{proposition}

\begin{proof}
[Proof of Proposition \ref{PropReg}]We will use Lemma \ref{LemTM} to finish
the proof. We propose%
\[
h(s)=[\rho+b-s]I(s<b),
\]
for some $\rho>0$. First we note that%
\begin{align}
\label{drift} &  \left.  E^{Q}(b-s-X;X\in(\Lambda^{-1}\left(  \Lambda\left(
b-s\right)  -a_{\ast\ast}\right)  ,b-s])\right. \\
&  =p_{\ast\ast}\frac{P\left(  X\in(\Lambda^{-1}\left(  \Lambda\left(
b-s\right)  -a_{\ast\ast}\right)  ,b-s]\right)  }{P\left(  X>\Lambda
^{-1}\left(  \Lambda\left(  b-s\right)  -a_{\ast\ast}\right)  \right)  }
\times\left.  E(b-s-X|X\in(\Lambda^{-1}\left(  \Lambda\left(  b-s\right)
-a_{\ast\ast}\right)  ,b-s])\right.  .\nonumber
\end{align}
Recall that%
\begin{equation}
\label{p*}p_{\ast\ast}=\min\{2\theta\bar{F}(b-s)/G(b-s),\widetilde
{\varepsilon}\}=\frac{2\theta(\iota-1)}{b-s}\left(  1+o\left(  1\right)
\right)
\end{equation}
as $b-s\nearrow\infty$, where $\theta=-\mu(1-\delta_{0}^{\ast})/(1+\delta
_{0}^{\ast})^{5}$. Therefore, we can select $\eta_{*}>0$ large enough so that
if $b-s\geq\eta_{*}$%
\[
-\frac{2\mu(\iota-1)(1-\delta_{0}^{\ast})^{2}}{\left(  b-s\right)
(1+\delta_{0}^{\ast})^{5}}\leq p_{\ast\ast}\leq-\frac{2\mu(\iota-1)}{\left(
b-s\right)  }.
\]
Now, note that $\eta_{*}$ can be chosen sufficiently large so that if
$a=e^{-2a_{\ast\ast}/\iota}$, then%
\[
\exp\left(  -\Lambda\left(  b-s\right)  +\Lambda\left(  a\left(  b-s\right)
\right)  \right)  =\frac{P\left(  X>b-s\right)  }{P\left(  X>a(b-s\right)
)}\leq\exp\left(  -a_{\ast\ast}\right)
\]
as long as $b-s\geq\eta_{*}$. Therefore,
\[
X\geq\Lambda^{-1}\left(  \Lambda\left(  b-s\right)  -a_{\ast\ast}\right)
\]
implies $X\geq a(b-s)$ and we have that%
\begin{equation}
\label{drift1}\left.  E(b-s-X|X\in(\Lambda^{-1}\left(  \Lambda\left(
b-s\right)  -a_{\ast\ast}\right)  ,b-s])\right.  \leq\left(  1-a\right)
\left(  b-s\right)  .
\end{equation}
Together with \eqref{drift}, \eqref{p*}, and \eqref{drift1}, if $b-s\geq
\eta_{*}$ we obtain%
\begin{align*}
E^{Q}(b-s-X;X\in(\Lambda^{-1}\left(  \Lambda\left(  b-s\right)  -a_{\ast\ast
}\right)  ,b-s]) \leq2|\mu|\left(  1-a\right)  (\iota-1).
\end{align*}
The previous estimates imply that by choosing $\eta_{*}>0$ large enough we can
guarantee that for all $b-s\geq\eta_{*}$ we have
\begin{align*}
&  E^{Q}(h(s+X))\\
&  =E^{Q}(\rho+b-s-X;s+X\leq b)\\
&  \leq(1-Q(X>b-s))(\rho +b-s-\mu+o(1))+2|\mu|(1-a) (\iota-1)\\
&  =(1-p_{\ast\ast}e^{-a_{\ast\ast}})(h(s)  -\mu+o(1))+2|\mu|(1-a)  (\iota-1).
\end{align*}
By noting that $\theta\leq|\mu|$, if $b-s\geq\eta_{*}$ and $\eta_{*}$ is
selected large enough we obtain that
\begin{align*}
&  E^{Q}(h(s+X))\\
&  \leq h(s)-\mu-p_{\ast\ast}e^{-a_{\ast\ast}}h\left(  s\right)
+2|\mu|\left(  1-a\right)  (\iota-1) +o(1)\\
&  \leq h(s)-\mu+\frac{2\mu(\iota-1)(1-\delta_{0}^{\ast})^{2}}{(1+\delta
_{0}^{\ast})^{5}}e^{-a_{\ast\ast}}\\
&  +\frac{2\mu(\iota-1)(1-\delta_{0}^{\ast})^{2}}{(1+\delta_{0}^{\ast}%
)^{5}\left(  b-s\right)  }e^{-a_{\ast\ast}}\rho-2\mu\left(  1-a\right)
(\iota-1) +o(1).
\end{align*}
The above inequality holds for all $\rho>0$ provided that $b-s\geq\eta
_{*}=G^{-1}\left(  \kappa^{-1/2}\right)  $ so that $b-s>\eta_{*}$ if and only
if $g\left(  s\right)  <1$. Since $\iota> 1.5$, one can choose $a_{**}$ and
$\delta_{0}^{*}$ sufficiently small such that
\[
\delta_{2}^{\ast}=2(\iota-1)\frac{(1-\delta_{0}^{\ast})^{2}}{\left(
1+\delta_{0}^{\ast}\right)  ^{5}}e^{-a_{\ast\ast}}-1-2(1-e^{-2a_{\ast\ast
}/\iota})\left(  \iota-1\right)  >0
\]
we conclude that%
\[
E^{Q}(h(s+X))\leq h\left(  s\right)  +\mu\delta_{2}^{\ast}%
\]
as long as $g\left(  s\right)  <1$. Now, if $g\left(  s\right)  =1$ (i.e. if
$0\leq b-s<\eta_{*}$) we do not apply the change of measure and therefore%
\begin{align*}
E^{Q}(h(s+X))  &  =E[ \rho+b-s-X  ;X\leq b-s]\\
&  \leq h(s)-E(X|X<0)-\rho P\left(  X>\eta_{*}\right)  .
\end{align*}
Given the selection of $\kappa$ (and therefore of $\eta_{*}=G^{-1}\left(
\kappa^{-1/2}\right)  $), we can choose $\rho$ large such that%
\[
-E(X|X<0)-\rho P\left(  X>\eta_{*}\right)  \leq\mu\delta_{2}^{\ast}<0.
\]
Hence,%
\[
E^{Q}\tau_{b}<h(0)/\left\vert \mu\right\vert \delta_{2}^{\ast}.
\]
Thereby, the conclusion of Lemma \ref{LemTM} follows by redefining the constants.
\end{proof}

\bigskip


\begin{remark}
\label{RemOp3_2}The previous result concerning the condition $\iota>1.5$
raises a couple of natural questions. First, what is special about a tail
index $\iota = 1.5$? What would be required in order to obtain both strong
efficiency and $E^{Q}\tau_{b}=O\left(  b\right)  $ assuming only $\iota>1$? We
believe that the previous result is basically optimal. We do not pursue this
claim with full rigor here but provide an argument showing why we expect this
to be the case. First, Theorem \ref{THAK} implies the approximation%
\[
P\left(  b\delta n<\tau_{b}\leq b\delta(n+1)b|\tau_{b}<\infty\right)
=[P\left(  Y_{0}>\delta|\mu|n(\iota-1)\right)  -P\left(  Y_{0}>\delta|\mu|(n+1)(\iota-1)\right)
](1+o\left(  1\right)  )
\]
as $b\nearrow\infty$ for any $\delta>0$. Even if we could apply importance
sampling directly to $\tau_{b}$ (rather than doing it through the $X_{j}$'s)
it would be reasonable to select $Q\left(  \cdot\right)  $ so that%
\[
Q\left(  b\delta n<\tau_{b}\leq b\delta(n+1)\right)  =c_{1}\left(
\delta\right)  n^{-\gamma_{1}}(1+o\left(  1\right)  )
\]
as $b\nearrow\infty$. Since we wish to have $E^{Q}\tau_{b}<\infty$ we should
impose the constraint $\gamma_{1}>2$. Now, we have that%
\[
P\left(  Y_{0}>\delta|\mu|n(\iota-1)\right)  -P\left(  Y_{0}>\delta|\mu|(n+1)(\iota-1)\right)
=\delta|\mu|(\iota-1)\left (1+{\delta|\mu|n}\right)^{-\iota}(1+o\left(  1\right)  )
\]
as $n\nearrow\infty$. On the other hand, strong efficiency imposes the
constraint that%
\begin{equation}
\sum_{n=1}^{\infty}\left(  \frac{P\left(  Y_{0}>\delta(n+1)\right)  -P\left(
Y_{0}>\delta n\right)  }{Q\left(  b\delta n<\tau_{b}\leq b\delta(n+1)\right)
}\right)  ^{2}Q\left(  b\delta n<\tau_{b}\leq b\delta(n+1)\right)
<\infty\label{SEC}%
\end{equation}
which suggests%
\begin{equation}\label{OPT}
\sum_{n=1}^{\infty}n^{-2\iota+\gamma_{1}}<\infty.
\end{equation}
Consequently, we also must have $2\iota>\gamma_{1}+1$. Combined with the
previous constraint (i.e. $\gamma_{1}>2$), it yields $\iota>3/2$.
\end{remark}

\bigskip

We will show that if $\iota>1$ we can control $1+\gamma$ relative moments (for
$\gamma$ small enough) and still keep $E^{Q}\tau_{b}=O\left(  b\right)  $.
However, before we do so, in order to complete the argument for the proof of
Theorem \ref{ThmSEb} we will continue working with $\gamma=1$ in the context
of Assumption B.

\begin{proposition}
\label{PropWeibull} If Assumptions B1-3 hold, we assume there exists
$\delta>0$ and $\beta\in\lbrack0,\beta_{0}]$ such that $\lambda(x)\geq\delta
x^{\beta-1}$ for $x$ sufficiently large. Then, there exist $a_{\ast}$,
$a_{\ast\ast}$, $p_{\ast}$, $p_{\ast\ast}$, $p_{j}$, $j=1,...,k$, such that
Theorem \ref{ThmSE1} holds and, in addition,
\[
E^{Q}\tau_{b}\leq\rho_{0}+\rho_{1}b^{1-\beta}.
\]
for $\rho_{0}$ and $\rho_{1}$ sufficiently large.
\end{proposition}

\begin{proof}
[Proof of Proposition \ref{PropWeibull}]Let $\beta \in \left( 0,\beta _{0}\right) $ and consider the Lyapunov
function,
\begin{equation*}
h(s)=[\rho +(b-s)^{1-\beta }]I(s<b).
\end{equation*}%
For all $\varepsilon >0$,%
\begin{align*}
& E^{Q}(h(s+X)) \\
& \leq Q\Big(X\leq (1-\varepsilon )(b-s)\Big)E^{Q}\Big(\rho
+(b-s-X)^{1-\beta }|X\leq (1-\varepsilon )(b-s)\Big) \\
& +(\rho +\varepsilon ^{1-\beta }(b-s)^{1-\beta })Q\left( (1-\varepsilon
)(b-s)\leq X\leq b-s\right) .
\end{align*}%
With Assumptions B1-3, if $\beta =0$, using L'Hopital rule on a subsequence,
we have%
\begin{equation*}
\underline{\lim }_{x\rightarrow \infty }\frac{x\bar{F}(x)}{G(x)}=\underline{%
\lim }_{x\rightarrow \infty }\frac{-\bar{F}(x)+x\lambda (x)\bar{F}(x)}{\bar{F%
}(x)}=\infty ;
\end{equation*}%
if $\beta \in (0,\beta _{0})$,%
\begin{equation*}
\underline{\lim }_{x\rightarrow \infty }\frac{x^{1-\beta }\bar{F}(x)}{G(x)}=%
\underline{\lim }_{x\rightarrow \infty }x^{1-\beta }\lambda (x)- (1-\beta)x^{-\beta
}\geq \delta .
\end{equation*}%
There exists $\varepsilon ,\delta ^{\prime }>0$ small enough and $\eta
_{\ast }$ sufficiently large such that for all $b-s>\eta _{\ast }$ and all $%
\rho >0$,
\begin{align*}
& E^{Q}(h(s+X)) \\
& \leq (1-2\theta \delta (b-s)^{\beta -1})\Big(\rho +(b-s)^{1-\beta
}-(1+\delta ^{\prime })(1-\beta )(b-s)^{-\beta }\mu \Big) \\
& +2\theta \delta (\rho +\varepsilon ^{1-\beta }(b-s)^{1-\beta
})(b-s)^{\beta -1} \\
& \leq (1-2\theta \delta (b-s)^{\beta -1})\Big(h(s)-(1+\delta ^{\prime
})(1-\beta )(b-s)^{-\beta }\mu \Big) \\
& +2\theta \delta (\rho +\varepsilon ^{1-\beta }(b-s)^{1-\beta
})(b-s)^{\beta -1} \\
& 
\leq h(s)-\theta \delta .
\end{align*}%
The above derivation is true for all $\beta >0$ satisfying conditions in the
proposition. When $\beta =0$ due to Assumption B1, one can always choose $\delta$ large such that $2\theta\delta >3|\mu |$. This allows us to control the contribution of the term $(1+\delta ^{\prime })(1-\beta )(b-s)^{-\beta }\mu$ in the above display. Therefore, this derivation is true
for all $\beta \in \lbrack 0,\beta _{0}]$.

On the other hand, if $b-s\leq\eta_{\ast}$ and we select $%
\eta_{\ast}=G^{-1}(\kappa^{-1/2})$ so that $g\left( s\right) <1$ if and only
if $b-s>\eta_{\ast}$, we obtain that%
\begin{align*}
& E^{Q}h(s+X) \\
& =Eh(s+X)\leq\rho+\left( b-s\right) ^{1-\beta}-\rho P\left( X>\eta_{\ast
}\right) +E(\left( b-s-X\right) ^{1-\beta}-\left( b-s\right) ^{1-\beta
};X\leq b-s).
\end{align*}
Clearly, once $\eta_{\ast}$ has been selected we can pick $\rho$ large
enough so that%
\begin{equation*}
-\rho P\left( X>\eta_{\ast}\right) +\sup_{0\leq b-s\leq\eta_{\ast}}E(\left(
b-s-X\right) ^{1-\beta}-\left( b-s\right) ^{1-\beta};X\leq b-s)\leq
-\delta/2.
\end{equation*}
Therefore,
\begin{equation*}
E^{Q}(h(s+X))\leq h(s)-\delta/2
\end{equation*}
and we conclude the result by applying Lemma \ref{LemTM}.
\end{proof}

\bigskip

\begin{proof}
[Proof of Theorem \ref{ThmSEb}]The conclusion follows immediately from
Propositions \ref{PropReg} and \ref{PropWeibull}.
\end{proof}

\bigskip

Finally, we come back to the problem of controlling $(1+\gamma)$-th moments in
order to guarantee $E^{Q}\tau_{b}=O\left(  b\right)  $ when $\bar F$ is
regularly varying with $\iota>1$. This corresponds to Theorem
\ref{Thm1_pls_Gamma}. The next proposition is central to the proof.

\begin{proposition}
\label{PropReg1_g} Suppose that Assumption A holds and that $\iota\in (1,1.5]$. Then,
we can choose $a_{\ast}$, $a_{\ast\ast}$, $p_{\ast}$, and $p_{\ast\ast}$, such that for each $\gamma\in(0,(\iota-1)/(2-\iota))$ there exists a $K>0$,
\[
E^{Q}Z^{1+\gamma}\leq Ku\left(  b\right)  ^{1+\gamma}%
\]
and $E^{Q}\tau_{b}=O\left(  b\right)  $ as $b\rightarrow\infty$.
\end{proposition}

\begin{remark}\label{RemGamma}
With a very similar argument as in Remark \ref{RemOp3_2}, we believe that the bound $1+(\iota-1)/(2-\iota)$ is the highest moment that one can control while maintaining $O(b)$ expected termination time. An analogous constraint to \eqref{OPT} is that
$$\sum_{n=1}^{\infty} n^{-(1+\gamma)(\iota-\gamma_1)-\gamma_1}<\infty.$$
This implies that $\gamma < (\iota -1)/(\gamma_1 -\iota)\leq (\iota -1)/(2 -\iota)$. Note that it is necessary to impose $\gamma_1 >2$ to have $O(b)$ expected termination time.
\end{remark}

\begin{proof}
[Proof of Proposition \ref{PropReg1_g}]The strategy is completely analogous to
the case of $\gamma=1$. We define
\[
g_{\gamma}\left(  s\right)  =\min\{\kappa G\left(  b-s\right)  ^{1+\gamma
},1\}.
\]
We need to verify the Lyapunov inequality only on $g_{\gamma}\left(  s\right)
<1$ (as before the case $g_{\gamma}\left(  s\right)  =1$ is automatic). We
select%
\[
p_{\ast\ast}=\min\{\theta\partial g_{\gamma}\left(  s\right)  /g_{\gamma
}\left(  s\right)  ,\tilde \varepsilon\}
\]
for $\tilde \varepsilon$ sufficiently small. Applying Lemma \ref{LemLyap} we need
to show that%
\begin{equation}
\frac{J_{\ast}}{\left(  1-p_{\ast\ast}\right)  ^{\gamma}}+\frac{J_{\ast\ast}%
}{p_{\ast\ast}^{\gamma}}\leq1, \label{LI1}%
\end{equation}
where $J_{\ast}$ and $J_{\ast\ast}$ are redefined as
\begin{align*}
J_{\ast}  &  =P(X\leq b-s-\Lambda^{-1}(\Lambda(b-s)-a_{\ast\ast}))^{\gamma
}E\left(  \frac{g_{\gamma}(s+X)}{g_{\gamma}(s)};X\leq b-s-\Lambda^{-1}%
(\Lambda(b-s)-a_{\ast})\right) \\
J_{\ast\ast}  &  =P(X>\Lambda^{-1}(\Lambda(b-s)-a_{\ast\ast}))^{\gamma
}E\left(  \frac{g_{\gamma}(s+X)}{g_{\gamma}(s)};X>\Lambda^{-1}(\Lambda
(b-s)-a_{\ast\ast})\right)  .
\end{align*}
Note that the $J_i$ terms analogous to \eqref{Ji} and \eqref{Jk} are all zero.
At the same time, we need to make sure that we can find $\rho>0$ such that if
\[
h\left(  s\right)  =\left[  \rho+(b-s)\right]  I\left(  b-s>0\right)
\]
then
\begin{equation}
E^{Q}h\left(  s+X\right)  \leq h\left(  s\right)  -\varepsilon\label{LI2}%
\end{equation}
for some $\varepsilon>0$ if $b>s$.

Inequality (\ref{LI1}) can be obtained following the same steps as we did in
I) to IV) in the previous section. First we note that if $\eta_{\ast}%
=G^{-1}\left(  \kappa^{-1/(1+\gamma)}\right)  $ is large enough (or
equivalently $\kappa$ is sufficiently large)
\[
\frac{J_{\ast\ast}}{p_{\ast\ast}^{\gamma}}\leq\frac{P(X>\Lambda^{-1}%
(\Lambda(b-s)-a_{\ast\ast}))^{\gamma+1}}{g\left(  s\right)  p_{\ast\ast
}^{\gamma}}=\frac{e^{a_{\ast\ast}(\gamma+1)}\overline{F}\left(  b-s\right)
}{\kappa(1+\gamma)^{\gamma}\theta^{\gamma}G\left(  b-s\right)  }.
\]
Also, for any $\delta>0$ we can ensure that if $\eta_{\ast}$ is large enough
and if $b-s>\eta_{\ast}$ then
\[
\frac{\theta\left(  1+\gamma\right)  \left(  \iota-1\right)  \left(
1-\delta\right)  }{b-s}\leq p_{\ast\ast}=\frac{\theta(1+\gamma)\overline
{F}\left(  b-s\right)  }{G\left(  b-s\right)  }\leq\frac{\theta\left(
1+\gamma\right)  \left(  \iota-1\right)  \left(  1+\delta\right)  }{b-s}%
\]
and we also can ensure that%
\[
\frac{J_{\ast}}{\left(  1-p_{\ast\ast}\right)  ^{\gamma}}\leq(1+\gamma
(1+\delta)p_{\ast\ast})E\left(  \frac{g_{\gamma}(s+X)}{g_{\gamma}(s)};X\leq
b-s-\Lambda^{-1}(\Lambda(b-s)-a_{\ast})\right)  .
\]
A similar development to that of Proposition \ref{PropJReg} yields that
$\eta_{\ast}$ can be chosen so that if $b-s>\eta_{\ast}$,
\[
E\left(  \frac{g_{\gamma}(s+X)}{g_{\gamma}(s)};X\leq b-s-\Lambda^{-1}%
(\Lambda(b-s)-a_{\ast})\right)  \leq1+\mu(1-\delta)\frac{\partial g_{\gamma
}(s)}{g_{\gamma}\left(  s\right)  }.
\]
Therefore,
\begin{align*}
\frac{J_{\ast}}{\left(  1-p_{\ast\ast}\right)  ^{\gamma}}  &  \leq\left(
1+\mu(1-\delta)\frac{\partial g_{\gamma}(s)}{g_{\gamma}\left(  s\right)
}\right)  (1+\gamma(1+\delta)p_{\ast\ast})\\
&  =\left(  1+\mu(1-\delta)\frac{\partial g_{\gamma}(s)}{g_{\gamma}\left(
s\right)  }\right)  \left(  1+\gamma(1+\delta)\frac{\theta\partial g_{\gamma
}(s)}{g_{\gamma}\left(  s\right)  }\right)
\end{align*}
and then
\begin{align*}
&  \frac{J_{\ast}}{\left(  1-p_{\ast\ast}\right)  ^{\gamma}}+\frac{J_{\ast
\ast}}{p_{\ast\ast}^{\gamma}}\\
&  \leq\left(  1+\frac{\mu(1-\delta)(1+\gamma)\overline{F}\left(  b-s\right)
}{G\left(  b-s\right)  }\right)  \left(  1+\frac{\theta\gamma(1+\delta
)(1+\gamma)\overline{F}\left(  b-s\right)  }{G\left(  b-s\right)  }\right) \\
&  +\frac{e^{a_{\ast\ast}(\gamma+1)}}{\kappa(1+\gamma)^{\gamma}\theta^{\gamma
}}\times\frac{\overline{F}\left(  b-s\right)  }{G\left(  b-s\right)  }.
\end{align*}
We then can select $\theta=\left\vert \mu\right\vert (1-\delta)^{2}%
/[\gamma(1+\delta)]$, $a_{\ast\ast}<\delta$ and $\kappa$ sufficiently large such that the right hand side the above display is less than one.
At the same time, the analysis required to enforce (\ref{LI2}) is similar to
that of Proposition \ref{PropReg}. We, therefore, omit the details. The key
fact is now that%
\[
-\frac{(1+\gamma)\mu(\iota-1)(1-\delta)^{3}}{\gamma\left(  b-s\right)
(1+\delta)}\leq p_{\ast\ast}%
\]
and now we need to enforce%
\[
\delta_{2}^{\ast}=\frac{(1+\gamma)(\iota-1)(1-\delta)^{3}}{\gamma\left(
1+\delta\right)  }e^{-a_{\ast\ast}}-1-(1+\gamma)(1-a)\left(  \iota-1\right)
>0,
\]
where $a=e^{-2a_{\ast\ast}/\iota}$. This can always be done if we choose
$\gamma<(\iota-1)/(2-\iota)$ and $\delta,a_{\ast\ast}>0$ sufficiently small.
\end{proof}

\bigskip

Now we provide the proof of Theorem \ref{Thm1_pls_Gamma}.

\begin{proof}
[Proof of Theorem \ref{Thm1_pls_Gamma}]From the result in Proposition
\ref{PropReg1_g}, the $(1+\gamma)$-th moment of the estimator and $E^{Q}%
\tau_{b}$ is properly controlled. We need to bound the total computation time
to achieve prescribed relative accuracy. Let $W_{1},W_{2},...$ be a sequence
of non-negative i.i.d. random variables with unit mean and suppose that
$EW_{i}^{1+\gamma}\leq K$ for $\gamma>0$. Define $R_{n}=(W_{1}+W_{2}%
+...+W_{n})/n$ and note that%
\[
P\left(  \left\vert R_{n}-1\right\vert \geq\varepsilon\right)  \leq P\left(
\left\vert R_{n}-1\right\vert \geq\varepsilon,\max_{i\leq n}W_{i}\leq
n\right)  +P\left(  \max_{i\leq n}W_{i}>n\right)  .
\]
Now using Chebyshev's inequality we have that%
\[
P\left(  \max_{i\leq n}W_{i}>n\right)  \leq nP\left(  W_{1}>n\right)
\leq\frac{K}{n^{\gamma}}.
\]
On the other hand, given $\max_{i\leq n} W_{i} < n$, $W_{i}$'s are still
i.i.d. and
\[
P\left(  \left\vert R_{n}-1\right\vert \geq\varepsilon\Big |\max_{i\leq
n}W_{i}\leq n\right)  \leq\frac{E\left(  W_{i}^{2}|W_{i}\leq n\right)
+o(1)}{n\varepsilon^{2}}=\frac{E\left(  W_{i}^{2}I(W_{i}\leq n\right)
)+o(1)}{n\varepsilon^{2}P\left(  W_{i}\leq n\right)  }.
\]
The $o(1)$ term in the above display is in fact $(E(W_{i} | W_{i} \leq
n)-1)^{2}$. Then, we have that for $\gamma\in\left(  0,1\right)  $%
\begin{align*}
E\left(  W_{i}^{2}I(W_{i}\leq n\right)  )  &  =2E\left(  I\left(  W_{i}\leq
n\right)  \int_{0}^{W_{i}}tdt\right) \\
&  \leq2\int_{0}^{n}tP\left(  W_{i}>t\right)  dt\leq2K\int_{0}^{n}\frac
{1}{t^{\gamma}}dt=\frac{2K}{1-\gamma}n^{1-\gamma}.
\end{align*}
Therefore, for $n$ sufficiently large we have that%
\[
P\left(  \left\vert R_{n}-1\right\vert \geq\varepsilon,\max_{i\leq n}W_{i}\leq
n\right)  \leq\frac{3K}{\left(  1-\gamma\right)  \varepsilon^{2}n^{\gamma}}.
\]
Thus, we have that%
\[
P\left(  \left\vert R_{n}-1\right\vert \geq\varepsilon\right)  \leq\frac
{3K}{\left(  1-\gamma\right)  \varepsilon^{2}n^{\gamma}}+\frac{K}{n^{\gamma}%
}\leq\frac{4K}{\left(  1-\gamma\right)  \varepsilon^{2}n^{\gamma}}.
\]
Applying these considerations to $W_{n}=Z_b/u\left(  b\right)  $ and letting
$4K/[(1-\gamma)\varepsilon^{2}n^{\gamma}]\leq\delta$ we obtain the conclusion
of the theorem.
\end{proof}

\section{Approximation in total variation and conditional limit theorems}

\label{SecTV}

\subsection{Approximation of the random walk up to $\tau_{b}$}

We will need the following lemma for the proof of approximation in total variation.

\begin{lemma}
\label{LemTV1}Let $Q_{0}$ and $Q_{1}$ be probability measures defined on the
same $\sigma$-field $\mathcal{F}$ such that $dQ_{1}=M^{-1}dQ_{0}$ for a
positive r.v. $M>0$. Suppose that for some $\varepsilon>0$, $E^{Q_{1}}\left(
M^{2}\right)  =E^{Q_{0}}M\leq1+\varepsilon$. Then,
\[
\sup_{A\in\mathcal{F}}\left\vert Q_{1}\left(  A\right)  -Q_{0}\left(
A\right)  \right\vert \leq\varepsilon^{1/2}.
\]

\end{lemma}

\begin{proof}
[Proof of Lemma \ref{LemTV1}]Note that%
\begin{align*}
\left\vert Q_{1}\left(  A\right)  -Q_{0}\left(  A\right)  \right\vert  &
=\left\vert E^{Q_{1}}\left(  1-M;A\right)  \right\vert \\
&  \leq E^{Q_{1}}\left(  \left\vert M-1\right\vert \right)  \leq[E^{Q_{1} }(
M-1) ^{2}]^{1/2}=\left(  E^{Q_{1}}M^{2}-1\right)  ^{1/2} \leq\varepsilon
^{1/2}.
\end{align*}

\end{proof}

Also, it is not hard to verify that by letting $P^{(b)}(\cdot)= P(\cdot
|\tau_{b} <\infty)$ we have
\[
\frac{dP^{(b)}}{dQ}=\frac{Z_{b}}{P(\tau_{b} <\infty)}.
\]
Then, it is sufficient to show that for $\varepsilon$ arbitrarily small there
exists $b$ sufficiently large depending on $\varepsilon$,
\[
E^{Q}Z_{b}^{2}<(1+\varepsilon)u^2(b).
\]

\begin{theorem}
\label{ThmTV} Suppose that Assumption A or B1-B3 hold. For any $\varepsilon
>0$, there exists $\eta_{*}>0$ such that for all $b>\eta_{*}$, there exists a
choice of $p_{\ast}$, $p_{\ast\ast}$, $p_{j}$, $j=1,\dots,k$ such that the
corresponding estimator $L_{b}$ satisfies,
\begin{equation}
E^{Q}Z_{b}^{2}\leq(1+\varepsilon)u^{2}(b). \label{BdTV}%
\end{equation}
Therefore, the importance sampling distribution converges in total variation
to the conditional distribution of the random walk given $\{\tau_{b}<\infty
\}$, as $b\rightarrow\infty$.
\end{theorem}

\begin{proof}
[Proof of Theorem \ref{ThmTV}]Given $\varepsilon,\varepsilon^{\prime}>0$ small, we
consider $\kappa>0$ and functions
\begin{align*}
\gamma(s)  &  =\left\{
\begin{array}
[c]{c}%
1+5\varepsilon+\kappa s^{1+\varepsilon^{\prime}}/b^{1+\varepsilon^{\prime}%
},s>0\\
1+5\varepsilon,s\leq0
\end{array}
\right.  \quad\\
g(s)  &  =\min\{1,\mu^{-2}\gamma(s)G^{2}(b-s)\}.
\end{align*}
Let $\eta_{\ast}=\sup\{b-s:g(s)=1\}.$ We can easily see that $\eta_{\ast
}\rightarrow\infty$ as $\kappa\rightarrow\infty.$ Also,
\[
1+5\varepsilon\leq\gamma(s)\leq\kappa+1+5\varepsilon,
\]
for all $s\leq b$. We proceed with a similar development as in the previous
section. We adopt the same notation as in (\ref{J*}), (\ref{J**}), (\ref{Ji}),
and (\ref{Jk}). Since $\gamma(s)$ is bounded, results as in Propositions
\ref{PropMid}, \ref{PropFirst}, and \ref{PropLast} still hold. In addition, we
can choose $a_{\ast\ast}$ small enough such that
\[
J_{\ast\ast}\leq\frac{P^{2}(X>\Lambda^{-1}(\Lambda(b-s)-a_{\ast\ast}%
))}{p_{\ast\ast}g(s)}\leq(1+\varepsilon)\frac{\bar{F}^{2}(b-s)}{p_{\ast\ast
}g(s)}.
\]
There is one last term, namely $J_{\ast}$. Note that
\[
\frac{g(s+X)}{g(s)}=\frac{G^{2}(s+X)}{G^{2}(s)}+\frac{G^{2}(s+X)}{G^{2}%
(s)}\left(  \frac{\gamma(s+X)}{\gamma(s)}-1\right)  .
\]
According to the proof of Proposition \ref{PropJReg} (more specifically
(\ref{DCGT})),
\[
E\left(  \frac{G^{2}(s+X)}{G^{2}(s)};X\leq b-s-\Lambda^{-1}(\Lambda
(b-s)-a_{\ast})\right)  \leq1+(2\mu+o(1))\bar{F}(b-s)/G(b-s).
\]
as $b-s\longrightarrow\infty$. Now, we consider the term
\[
E\left(  \frac{G^{2}(s+X)}{G^{2}(s)}\left(  \frac{\gamma(s+X)}{\gamma
(s)}-1\right)  ;X\leq b-s-\Lambda^{-1}(\Lambda(b-s)-a_{\ast})\right)  .
\]
For all $b\geq s>b^{\varepsilon^{\prime}}$ and $s+X>0$,
\[
\frac{G^{2}(s+X)}{G^{2}(s)}\left(  \frac{\gamma(s+X)}{\gamma(s)}-1\right)
=\kappa\gamma^{-1}(s)s^{1+\varepsilon^{\prime}}b^{-1-\varepsilon^{\prime}%
}\left(  \left(  1+X/s\right)  ^{1+\varepsilon^{\prime}}-1\right)  \frac
{G^{2}(s+X)}{G^{2}(s)}.
\]
Therefore, for $b\geq s>b^{\varepsilon^{\prime}}$, by dominated convergence,
\[
\gamma(s)E\left(  \frac{b^{1+\varepsilon^{\prime}}}{s^{\varepsilon^{\prime}}%
}\frac{G^{2}(s+X)}{G^{2}(s)}\left(  \frac{\gamma(s+X)}{\gamma(s)}-1\right)
;X\leq b-s-\Lambda^{-1}(\Lambda(b-s)-a_{\ast})\right)  \rightarrow
\kappa(1+\varepsilon^{\prime})\mu,
\]
as $b-s\rightarrow\infty$. For $s\leq b^{\varepsilon^{\prime}}$,
\begin{align*}
&  E\left(  \frac{G^{2}(s+X)}{G^{2}(s)}\left(  \frac{\gamma(s+X)}{\gamma
(s)}-1\right)  ;X\leq b-s-\Lambda^{-1}(\Lambda(b-s)-a_{\ast})\right) \\
&  =O(b^{-1-\varepsilon^{\prime}+\varepsilon^{\prime2}})=o(\bar{F}%
(b-s)/G(b-s))
\end{align*}
as $b\nearrow\infty$ uniformly over $s\leq b^{\varepsilon^{\prime}}$.
Consequently, it follows that
\[
E\left(  \frac{g(s+X)}{g(s)};X\leq b-s-\Lambda^{-1}(\Lambda(b-s)-a_{\ast
})\right)  \leq1+(2\mu+o(1))\bar{F}(b-s)/G(b-s),
\]
as $b-s\rightarrow\infty$. We choose,
\[
p_{\ast\ast}=\min\{\varepsilon,-(1-\varepsilon)\mu\bar{F}(b-s)/G(b-s)\},\quad
p_{j}=\varepsilon^{2}p_{\ast\ast}.
\]
To be consistent with the previous notations, we let
\begin{equation}\label{theta}
\theta = -\frac{\mu(1-\varepsilon)}{2}.
\end{equation}
Then,
\begin{align*}
E\left[  \frac{g(s+X)}{g(s)}r_{s}(X)\right]   &  \leq\left(  1+(1-\varepsilon
+o(\varepsilon))\mu\frac{\bar{F}(b-s)}{G(b-s)}\right)  ^{-1}\left[
1+(2\mu+o(1))\frac{\bar{F}(b-s)}{G(b-s)}\right] \\
&  +o(1)k\varepsilon^{-2}\bar{F}(b-s)/G(b-s)-(1+\varepsilon)\frac{\mu\bar
{F}(b-s)}{\gamma(s)G(b-s)(1-\varepsilon)}.
\end{align*}
When $s\leq b/2$,
\begin{align*}
E\left[  \frac{g(s+X)}{g(s)}r_{s}(X)\right]   &  \leq1-(1+o(\varepsilon
))\mu\frac{\bar{F}(b-s)}{G(b-s)}+(2\mu+o(1))\frac{\bar{F}(b-s)}{G(b-s)}\\
&  +o(1)k\varepsilon^{-2}\frac{\bar{F}(b-s)}{G(b-s)}-(1+3\varepsilon)\frac
{\mu\bar{F}(b-s)}{\gamma(s)G(b-s)}.
\end{align*}
Because $\gamma(s)\geq1+5\varepsilon$, for $b$ large enough, $E\left[
\frac{g(s+X)}{g(s)}r_{s}(X)\right]  \leq1$, when $s\leq b/2$. For $s\geq
b/2$,
\[
\gamma(s)\geq\kappa/4.
\]
Then
\begin{align*}
E\left[  \frac{g(s+X)}{g(s)}L(X)\right]   &  \leq1-(1+o(\varepsilon))\mu
\frac{\bar{F}(b-s)}{G(b-s)}+(2\mu+o(1))\frac{\bar{F}(b-s)}{G(b-s)}\\
&  +o(1)k\varepsilon^{-2}\frac{\bar{F}(b-s)}{G(b-s)}-\frac{4(1+3\varepsilon
)}{\kappa}\frac{\mu\bar{F}(b-s)}{G(b-s)}.
\end{align*}
For any $\varepsilon>0$ one can always choose $\kappa$ large enough such that
$E\left[  \frac{g(s+X)}{g(s)}r_{s}(X)\right]  \leq1$ when $s\geq b/2$ and
$g(s)<1$. Therefore,
\[
E^{Q}L^{2}\leq g(0)=(1+5\varepsilon)\mu^{-2}G(b)^{2},
\]
for $b$ large enough. The conclusion then follows from Lemma \ref{LemTV1} and Theorem \ref{ThmPV}.
\end{proof}

\begin{proof}
[Proof of Theorem \ref{ThmTVA}]The conclusion is a direct application of Lemma
\ref{LemTV1} and Theorem \ref{ThmTV}.
\end{proof}

\bigskip

Here we emphasize that the choices of parameters of the mixture family in the
current section are different from those in Section \ref{SecTermination}.
Especially for the regularly varying case with $\iota\in(1.5,2)$, in order to
have finite expected termination, we will have the importance sampling
distribution deviate from the zero-variance change of measure.

\subsection{Conditional central limit theorem\label{SectCLT}}

The goal of this section is to provide a functional approximation to the joint
distribution of
\[
\left\{  \left(  \tau_{b},S_{\left\lfloor u\tau_{b}\right\rfloor },S_{\tau
_{b}}\right)  :u\in\lbrack0,1)\right\}  ,
\]
conditional on $\{\tau_{b}<\infty\}$ as $b\rightarrow\infty$. To make the
discussion smooth, we postpone some technical proofs to Appendix
\ref{ApdTotal}.

For all the theorems so far, we assume either Assumption A or Assumptions
B1-B3. In this section, in the setting of Assumption B, we will further impose
Assumption B4.


The approximation will be obtained based on a coupling of two processes
governed according to a probability measure which shall be denoted by
$Q^{\ast}$. Our importance sampling
distribution induces a process that behaves most of the time like a regular
random walk, except that occasional large jumps occur with
probability $p_{**}$. We will couple this process with a regular random
walk and argue that with high probability as $b\nearrow\infty$ we have that
$\tau_{b}$ coincides precisely with the first of such large jumps.

We now proceed to formalize this intuition. Consider the process $\hat
{S}=\{\hat{S}_{n}:n\geq0\}$, where $\hat{S}_{n}=\hat{X}_{1}+...+\hat{X}_{n}$,
$\hat{S}_{0}=0$, and we have that%
\begin{equation}
Q^{\ast}(\hat{X}_{n+1}\in dx|\hat{S}_{n}=s) \triangleq q_{s}\left(  x\right)
dx=r_{s}^{-1}(x)f(x)dx. \label{Dxtil1}%
\end{equation}
The function $r_{s}^{-1}(x)$ is chosen to satisfy the conditions of Theorem
\ref{ThmTV}. We shall slightly abuse notation by letting $\tau_{b}%
=\inf\{n:\hat{S}_{n}>b\}$.

We further introduce a random walk $\tilde{S}=\{\tilde{S}_{n}:n\geq1\}$ such
that $\tilde{S}_{n}=\tilde{X}_{1}+...+\tilde{X}_{n}$ and with the property
that the $\tilde{X}_{i}$'s are i.i.d. under $Q^{\ast}$ and have density
\begin{equation}
Q^{\ast}(\tilde{X}_{i}\in dx)=f(x)dx. \label{DXtil}%
\end{equation}

The joint law of $\hat{S}$ and $\tilde{S}$ will be described next.

We first define
\begin{equation}
p(s)=\frac{p_{\ast}I\left(  b-s>\eta_{\ast}\right)  }{P(X\leq b-s-\Lambda
^{-1}(\Lambda(b-s)-a_{\ast}))}+I\left(  b-s\leq\eta_{\ast}\right).
\label{DefpsD}%
\end{equation}
Note that by possibly increasing the selection of $\kappa$ and  $\eta_{\ast}=\sup\{b-s:g(s)=1\}$ in Theorem \ref{ThmTV}, we can always
guarantee that $p\left(  s\right)  \in\lbrack 0,1]$. Actually $p(s)\rightarrow 1$ as $b-s \rightarrow \infty$.
Next define
\begin{equation}
q_{s}^{\ast}(x)=I\left(  p\left(  s\right)  <1\right)  (1-p(s))^{-1}%
(q_{s}(x)-p(s)f(x)). \label{Mix}%
\end{equation}
The next lemma shows that $q_{s}^{\ast}(\cdot)$ is a density function and
provides a decomposition of $q_{s}\left(  x\right)  $ that will allow us to
describe the joint law of $\hat{S}$ and $\tilde{S}$. The proof of the lemma is
given in Appendix \ref{ApdTotal}.

\begin{lemma}
\label{LemDensDec}If $p\left(  s\right)  <1$ we have that $q_{s}^{\ast}%
(\cdot)$ is a density function provided that $\kappa$ (and therefore
$\eta_{\ast}$) are chosen large enough. We thus have the mixture decomposition%
\begin{equation}
q_{s}(x)=p(s)f(x)+(1-p(s))q_{s}^{\ast}(x). \label{decomp}%
\end{equation}

\end{lemma}

\bigskip

The processes $\hat{S}$ and $\tilde{S}$ evolve jointly as follows under
$Q^{\ast}$. First simply let $\tilde{S}$ evolve according to (\ref{DXtil}).
Now, at any given time $n+1$ the evolution of $\tilde{S}$ obeys the following
rule. Given that $\hat{S}_{n}=s$, $\hat{X}_{n+1}$ is constructed as follows.
First, we sample a Bernoulli random variable to choose among $f(\cdot)$ and
$q_{s}^{\ast}(\cdot)$ according to the probabilities $p(s)$ and $1-p(s)$
respectively. If $f(\cdot)$ has been chosen, we let $\hat{X}_{n+1}=\tilde
{X}_{n+1}$. Otherwise, we construct $\hat{X}_{n+1}$ from the $q_{s}^{\ast
}(\cdot)$ and $\tilde{X}_{n+1}$ from $f(x)$ independently. We further let
\[
N_{b}=\inf\{n\geq1:\tilde{X}_{n}\not =\hat{X}_{n}\},
\]
which is the first time that $f(x)$ is not chosen. We intend to show that
$P(N_{b}=\tau_{b})\rightarrow1$ as $b\rightarrow\infty$. The result is
summarized in the following lemmas and propositions whose proofs are given in
Appendix \ref{ApdTotal}.

\begin{lemma}
\label{LemN}%
\[
\lim_{b\rightarrow\infty}Q^{\ast}(N_{b}<\infty)=1.
\]

\end{lemma}

\begin{lemma}
\label{LemJump} Let $\varepsilon$ be chosen as in Theorem \ref{ThmTV}. There
exists $b_{0}>0$ (depending on $a_{**}$ and $\varepsilon$) and $\gamma(a_{\ast\ast},\varepsilon)>0$ such that
$\gamma(a_{\ast\ast},\varepsilon)\rightarrow0$ as $a_{\ast\ast}\rightarrow0$
and $\varepsilon\rightarrow0$, satisfying that%
\[
Q^{\ast}(\tau_{b}=N_{b})\geq1-\gamma(a_{\ast\ast},\varepsilon),
\]
for all $b>b_{0},$ where $\tau_{b}=\inf\{n\geq1:\hat{S}_{n}\geq b\}$.
\end{lemma}

Now, we are ready to present the result which uses $\tilde{S}$ to approximate
the process $\hat{S}$ up to time $\tau_{b}$.

\begin{proposition}
\label{PropCoup} There exists a family of sets $(B_{b}:b>0)$ such that
$P(B_{b})\rightarrow1$ as $b\rightarrow\infty$ and with the property that for
all $\tilde{S}\in B_{b}$%
\[
Q^{\ast}(N_{b}>ta(b)|\tilde{S})=P(Z_{\theta}>t|\mu|)(1+o(1)),
\]
as $b\rightarrow\infty$, where $a(x)=G(x)/\bar{F}(x)$ and $\theta$ is defined in \eqref{theta}.

\begin{itemize}
\item Under Assumption A,
\[
P(Z_{\theta}>t)=\left(  1+\frac{t}{\iota-1}\right)  ^{-\frac{2\theta(\iota
-1)}{|\mu|}},
\]
for all $t\geq0$.

\item Under Assumptions B1-4,
\[
P(Z_{\theta}>t)=e^{-\frac{2\theta t}{|\mu|}}.
\]

\end{itemize}
\end{proposition}

\begin{proof}
[Proof of Theorem \ref{ThmTVD}]Thanks to Theorem \ref{ThmTV}, the distribution of $\{\hat S_{n}: 1\leq n
\leq\tau_{b}\}$ under $Q^{*}$ converges in total variation to the
distribution of $\{S_{n}: 1\leq n \leq\tau_{b}\}$ given $\tau_{b}<\infty$
under $P$. It is sufficient to show the limit theorem of $\{\hat S_{n}:
1\leq n \leq\tau_{b}\}$ under $Q^{*}$.

Thanks to Proposition \ref{PropCoup}, we are able to construct a random
variable $Z_{\theta }$ following the distributions stated in Proposition \ref%
{PropCoup} such that $Z_{\theta }$ is independent of $\tilde{S}$ and
\begin{equation*}
\frac{N_{b}}{a(b)}-\frac{Z_{\theta }}{|\mu |}\rightarrow 0,
\end{equation*}%
almost surely as $b\rightarrow \infty $. Thanks to Lemma \ref{LemJump}, we
have that%
\begin{equation*}
\left( \frac{N_{b}}{a(b)},\left\{ \frac{\tilde{S}_{tN_{b}}-t\mu N_{b}}{\sqrt{%
N_{b}}}\right\} _{0\leq t<1},\frac{\hat{S}_{N_{b}}-b}{a(b)}\right) -\left(
\frac{\tau _{b}}{a(b)},\left\{ \frac{\hat{S}_{t\tau _{b}}-t\mu \tau _{b}}{%
\sqrt{\tau _{b}}}\right\} _{0\leq t<1},\frac{\hat{S}_{\tau _{b}}-b}{a(b)}%
\right) \rightarrow 0
\end{equation*}%
in probability as $b\rightarrow \infty $ (in fact, the convergence holds for
almost every $\tilde{S}$ in the sequence $B_{b}$). Further, as $b\rightarrow
\infty $,\ we can let $\theta \rightarrow -\mu /2$. So it is possible to
construct a random variable $Y_{0}$ independent of $\tilde{S}$ and following
distribution stated in the theorem such that
\begin{equation*}
Z_{\theta }\rightarrow Y_{0},
\end{equation*}%
almost surely as $b\rightarrow \infty $. Now, using a standard strong
approximation result (see for instance \cite{EKM97}) we can (possibly by
further enlarging the probability space) assume that
\begin{equation}
\tilde{S}_{\left\lfloor t\right\rfloor }=\mu t+\sigma B\left( t\right)
+e\left( t\right)   \label{S_q}
\end{equation}%
where $e\left( \cdot \right) $ is a (random) function such that%
\begin{equation*}
\frac{e\left( xt\right) }{t^{1/2}}\longrightarrow 0
\end{equation*}%
with probability one uniformly on compact sets on $x\geq 0$ as $t\nearrow
\infty $. Therefore, we have that%
\begin{equation*}
\frac{\tilde{S}_{tN_{b}}-t\mu N_{b}}{\sqrt{N_{b}}}=\frac{\sigma B\left(
ta(b)Y_{0}/|\mu |+ta(b)\xi _{b}\right) +e_{b}\left( ta(b)Y_{0}/|\mu
|+ta(b)\xi _{b}\right) }{\sqrt{a(b)Y_{0}/|\mu |+a(b)\xi _{b}}},
\end{equation*}%
where $\xi _{b}\rightarrow 0$ as $b\rightarrow \infty $. For $\delta $
arbitrarily small, we now verify that for each $z>\delta $,
\begin{equation*}
\sup_{0\leq u\leq 1}\left\vert \frac{B\left( ua(b)z+ua(b)\xi _{b}\right)
-B\left( ua(b)z\right) }{\sqrt{a(b)z}}\right\vert \longrightarrow 0,
\end{equation*}%
as $a(b)\rightarrow \infty $. Given $\xi _{b}\rightarrow 0$ in probability,
it suffices to bound the quantity%
\begin{equation*}
\sup_{u,s\in (0,1),\left\vert u-s\right\vert \leq \varepsilon /\delta
}\left\vert \frac{B\left( ua(b)z\right) -B\left( sa(b)z\right) }{\sqrt{a(b)z}%
}\right\vert .
\end{equation*}%
By the invariance principle the previous quantity equals in distribution to%
\begin{equation*}
\sup_{u,s\in (0,1),\left\vert u-s\right\vert \leq \varepsilon /\delta
}\left\vert B\left( u\right) -B\left( s\right) \right\vert ,
\end{equation*}%
which is precisely the modulus of continuity of Brownian motion evaluated $%
\varepsilon /\delta $. By continuity of Brownian motion, its modulus of
continuity goes to zero almost surely as $\varepsilon \longrightarrow 0$.
Consequently, we obtain%
\begin{equation*}
\left( \frac{N_{b}}{a(b)},\left\{ \frac{\tilde{S}_{tN_{b}}-t\mu N_{b}}{\sqrt{%
N_{b}}}\right\} _{0\leq t<1},\frac{\hat{S}_{N_{b}}-b}{a(b)}\right) -\left(
\frac{Y_{0}}{|\mu |},\left\{ \frac{\tilde{S}_{ta(b)Y_{0}/|\mu |}+ta(b)Y_{0}}{%
\sqrt{a(b)Y_{0}/|\mu |}}\right\} _{0\leq t<1},\frac{\hat{S}_{N_{b}}-b}{a(b)}%
\right) \Longrightarrow 0.
\end{equation*}%
Because $Y_{0}$ is independent of $\tilde{S}$, using the invariance
principle for Brownian motion, we have that
\begin{equation*}
\left( \frac{Y_{0}}{|\mu |},\left\{ \frac{\tilde{S}_{ta(b)Y_{0}/|\mu
|}+ta(b)Y_{0}}{\sqrt{a(b)Y_{0}/|\mu |}}\right\} _{0\leq t<1},\frac{\hat{S}%
_{N_{b}}-b}{a(b)}\right) \Rightarrow \left( \frac{Y_{0}}{|\mu |},\left\{
\sigma B(t)\right\} _{0\leq t<1},Y_{1}\right) .
\end{equation*}%
Now, we figure out the joint distribution between $Y_{0}$ and $Y_{1}$. Note
that $\hat{S}_{N_{b}}-b$ satisfies
\begin{equation*}
\frac{\hat{S}_{N_{b}}-b}{a(b)}=\frac{\hat{X}_{N(b)}+\tilde{S}_{N(b)-1}-b}{%
a(b)}.
\end{equation*}%
In turn, we have,
\begin{equation*}
\frac{\tilde{S}_{N\left( b\right) -1}}{a\left( b\right) }+Y_{0}\rightarrow 0
\end{equation*}%
in probability. In addition, the conditional distribution of $\hat{X}_{N(b)}$
given $\tilde{S}_{N\left( b\right) -1}$ is asymptotically (as $b\rightarrow
\infty $) that of $\tilde{X}$ given that $\tilde{X}>b-\tilde{S}_{N\left(
b\right) -1}$ and $\tilde{S}_{N\left( b\right) -1}$, where $\tilde{X}$ is a
random variable with density $f\left( \cdot \right) $ independent of $\tilde{%
S}_{N\left( b\right) -1}$. Therefore, the law of $(\hat{X}_{N\left( b\right)
}+\tilde{S}_{N\left( b\right) -1}-b)/a\left( b\right) $ given $\tilde{S}%
_{N\left( b\right) -1}$ can be approximated by that of $\tilde{X}/a\left(
b\right) -Y_{0}-b/a\left( b\right) $ given $Y_{0}$ and $\tilde{X}%
-Y_{0}a\left( b\right) >b$.

In the setting of Assumptions B1-B4, we establish in the proof of
Proposition \ref{PropCoup} that $a\left( b\right) =\left( 1+o\left( 1\right)
\right) /\lambda\left( b\right) $ as $b\nearrow\infty$. Because of
Assumption B1 we have that $a\left( b\right) =o\left( b\right) $. Because of
Assumption B4 we have that for each $y>0$
\begin{equation}
Q^*(\tilde{X}>ya\left( b\right) +Y_{0}a\left( b\right) +b|\tilde{X}%
>b+Y_{0}a\left( b\right) ,Y_{0})\rightarrow P\left( Y_{1}>y\right)
=\exp\left( -y\right)  \label{Ea}
\end{equation}
as $b\nearrow\infty$. Hence, $Y_1$ is an exponential random variable with
expectation one and is independent of $Y_0$.

Now, suppose that Assumption A holds. We have that $a\left( b\right)
=b/\left( \iota-1\right) +o\left( b\right) $ as $b\nearrow\infty$. Therefore,%
\begin{align*}
& Q^*(\tilde{X}-(Y_{0}a\left( b\right) +b)>ya\left( b\right) |\tilde{X}%
>Y_{0}a\left( b\right) +b,Y_{0}=y_0) \\
& =\left( 1+o\left( 1\right) \right) Q^*\left(\tilde{X}-(Y_{0}+\iota-1)a\left(
b\right) >ya\left( b\right) |\tilde{X}>(Y_{0} +\iota-1)a\left( b\right)
,Y_{0}=y_0\right) \\
& \longrightarrow P\left( W>y/(y_{0}+\iota-1)\right) ,
\end{align*}
where
\begin{equation*}
P\left( W>t\right) =(1+t)^{-\iota}
\end{equation*}
for $t\geq0$. Now we need to verify that the law of $\left(
Y_{0},Y_{1}\right) $ as stated in the theorem coincides with that of $%
(Y_{0},W[Y_{0}+(\iota-1)])$. First we note that the joint density of $\left(
Y_{0},Y_{1}\right) $ is given by%
\begin{equation*}
\frac{\partial^{2}}{\partial y_{0}\partial y_{1}}P\left(
Y_{0}>y_{0},Y_{1}>y_{1}\right) =\frac{\iota}{\iota-1}\left(
1+(y_{0}+y_{1})/\left( \iota-1\right) \right) ^{-\iota-1}.
\end{equation*}
Therefore,
\begin{equation*}
\frac{P\left( Y_{1}\in dy_{1}|Y_{0}=y_{0}\right) }{dy_{1}}\propto\left(
\iota-1+y_{0}+y_{1}\right) ^{-\iota-1}.
\end{equation*}
On the other hand,
\begin{equation*}
\frac{P(W[y_{0}+(\iota-1)]\in dy_{1})}{dy_{1}}=\iota\left(
1+y_{1}/[y_{0}+\left( \iota-1\right) ]\right) ^{-\iota-1}\propto\left(
\iota-1+y_{0}+y_{1}\right) ^{-\iota-1}.
\end{equation*}
The independence between $B(t)$ and $(Y_0,Y_1)$ is straightforward. This
concludes the proof of the theorem.
\end{proof}

\bigskip

\section{Implementation and examples}

\label{SecSim}

We implemented the algorithm and compare the performance with other existing
algorithms in literature. In particular, we investigated two cases: regularly
varying distribution and Weibull like distribution.

\paragraph{Regularly varying distribution.}

We consider the increment has the following representation.
\[
X_{i}=V_{i}-T_{i},
\]
where $V_{i}$ are i.i.d. with distribution that $P(V_{i}>v)=(1+v)^{-2.5}$ for
$v>0$ and $T_{i}$'s are i.i.d. exponential random variables with expectation
$4/3$. It is not hard to verify that $E(X_{i})=-2/3$. In fact, this
corresponds to the tail probability of the steady-state waiting time of an
$M/G/1$ queue. There are a few provably efficient algorithms in literature
including. Asmussen and Kroese (2006) (AK) \cite{AsmKro06}, and Dupuis, Leder
and Wang (2006) (DLW) \cite{DupLedWang07} proposed efficient rare-event
simulation estimators for geometric sums of regularly varying random
variables. Blanchet and Glynn (2008) (BG) \cite{BlaGly07}, and Blanchet,
Glynn, and Liu (2007) (BGL) \cite{BGL07} proposed estimators for the tail of
the steady state $G/G/1$ waiting time. Table 1 compares the performance of
these algorithms. We use BL to denote the algorithm proposed in the current
paper, with one cut-off point $c_{0}=0.9(b-s)$.

\begin{table}[ht]
\begin{center}%
\begin{tabular}
[c]{|c|c|c|c|}\hline
$%
\begin{array}
[c]{c}%
\text{\lbrack Estimation]}\\
\text{\lbrack Std. Error]}%
\end{array}
$ & $b=10^{2}$ & $b=10^{3}$ & $b=10^{4}$\\\hline
BL & $%
\begin{array}
[c]{c}%
1.047e-03\\
3.76e-05
\end{array}
$ & $%
\begin{array}
[c]{c}%
3.175e-05\\
2.602e-07
\end{array}
$ & $%
\begin{array}
[c]{c}%
9.877e-07\\
8.187e-09
\end{array}
$\\\hline
AK & $%
\begin{array}
[c]{c}%
1.199e-03\\
1.479e-05
\end{array}
$ & $%
\begin{array}
[c]{c}%
3.145e-05\\
2.186e-07
\end{array}
$ & $%
\begin{array}
[c]{c}%
9.980e-07\\
6.945e-09
\end{array}
$\\\hline
BG & $%
\begin{array}
[c]{c}%
1.079e-03\\
5.968e-06
\end{array}
$ & $%
\begin{array}
[c]{c}%
3.146e-05\\
9.725e-08
\end{array}
$ & $%
\begin{array}
[c]{c}%
9.980e-07\\
2.073e-09
\end{array}
$\\\hline
BGL & $%
\begin{array}
[c]{c}%
1.022e-03\\
3.835e-05
\end{array}
$ & $%
\begin{array}
[c]{c}%
3.167e-05\\
1.598e-06
\end{array}
$ & $%
\begin{array}
[c]{c}%
1.128e-06\\
7.280e-08
\end{array}
$\\\hline
DLW & $%
\begin{array}
[c]{c}%
1.046e-03\\
5.195e-06
\end{array}
$ & $%
\begin{array}
[c]{c}%
3.163e-05\\
1.694e-07
\end{array}
$ & $%
\begin{array}
[c]{c}%
9.905e-07\\
2.993e-09
\end{array}
$\\\hline
\end{tabular}
\end{center}
\caption{Estimated tail probabilities of regularly varying random walks}%
\label{TabReg}%
\end{table}

\paragraph{Weibull-type distribution}

For the Weibull-type case, we consider the increment to have the following
distribution,
\[
P(X>x) = e^{-2\sqrt{t+1}},
\]
for $t\geq-1$ and $EX_{i} = -\frac1 2$. Table \ref{TabWei} compares the
algorithm in this paper (BL) and that of Blanchet and Glynn (2008) (BG). For
the implementation, we choose that $c_{0} = \sqrt{b-s}$, $c_{1} = 0.1(b-s)$,
$c_{2} = 0.5(b-s)$, $c_{3} = 0.9(b-s)$, $c_{4} = b-s -\sqrt{b-s}$.

\begin{table}[ht]
\begin{center}%
\begin{tabular}
[c]{|c|c|c|c|}\hline
$%
\begin{array}
[c]{c}%
\text{\lbrack Estimation]}\\
\text{\lbrack Std. Error]}%
\end{array}
$ & $b=250$ & $b=500$ & $b=650$\\\hline
BL & $%
\begin{array}
[c]{c}%
6.985e-13\\
5.639e-14
\end{array}
$ & $%
\begin{array}
[c]{c}%
1.778e-18\\
1.936e-19
\end{array}
$ & $%
\begin{array}
[c]{c}%
3.900e-21\\
5.696e-22
\end{array}
$\\\hline
BG & $%
\begin{array}
[c]{c}%
7.076e-13\\
1.20e-14
\end{array}
$ & $%
\begin{array}
[c]{c}%
1.897e-18\\
5.083e-20
\end{array}
$ & $%
\begin{array}
[c]{c}%
3.971e-21\\
7.95e-23
\end{array}
$\\\hline
\end{tabular}
\end{center}
\caption{Estimated tail probabilities of the Weibull-type distribution}%
\label{TabWei}%
\end{table}

\appendix

\section{Technical proofs in Sections \ref{SecPre} and \ref{SectionAnalysis}}

\label{SecTech}

\begin{proof}
[Proof of Lemma \ref{LemRateFun}]Observe that B2 implies $\log(\Lambda\left(
x\right)  /\Lambda\left(  b_{0}\right)  )\leq\log((x/b_{0})^{\beta_{0}})$. In
other words, $\Lambda\left(  x\right)  \leq\Lambda\left(  b_{0}\right)
b_{0}^{-\beta_{0}}x^{\beta_{0}}$. Consequently, substituting into B2 we have
that for $x\geq b_{0}$%
\[
\lambda\left(  x\right)  \leq\beta_{0}\Lambda\left(  x\right)  /x\leq\beta
_{0}\Lambda\left(  b_{0}\right)  b_{0}^{-\beta_{0}}x^{\beta_{0}-1}=O\left(
x^{\beta_{0}-1}\right).
\]
\end{proof}

\bigskip

\begin{proof}
[Proof of Lemma \ref{LemIntTail}]First, since $G\left(  \cdot\right)  $ is
decreasing then for $x\leq b-\Lambda^{-1}(\Lambda(b)-a_{\ast})$
\[
\frac{G(b-x)}{G(b)}\leq\frac{G(\Lambda^{-1}\left(  \Lambda\left(  b\right)
-a_{\ast}\right)  )}{G(b)}.
\]
By continuity of $G\left(  \cdot\right)  $ it suffices to show that the right
hand side is bounded for all $b$ sufficiently large. Using L'Hopital's rule we
conclude that%
\[
\frac{G(\Lambda^{-1}\left(  \Lambda\left(  b\right)  -a_{\ast}\right)
)}{G(b)}\sim\frac{\exp\left(  -\Lambda\left(  b\right)  +a_{\ast}\right)
}{\exp\left(  -\Lambda\left(  b\right)  \right)  }\left.  \frac{d}{dx}%
\Lambda^{-1}\left(  \Lambda\left(  x\right)  -a_{\ast}\right)  \right\vert
_{x=b}.
\]
Now, note that for all $x\geq b_{0}$%
\[
\frac{d}{dx}\Lambda^{-1}\left(  \Lambda\left(  x\right)  -a_{\ast}\right)
=\frac{\lambda\left(  x\right)  }{\lambda\left(  \Lambda^{-1}\left(
\Lambda\left(  x\right)  -a_{\ast}\right)  \right)  }\leq\frac{\lambda\left(
x\right)  }{\lambda\left(  \Lambda^{-1}\left(  \Lambda\left(  x\right)
\right)  \right)  }=1.
\]
The inequality follows from the fact that $\lambda\left(  \cdot\right)  $ is
non increasing and $a_{\ast}>0$. This allows to conclude the statement of the lemma.
\end{proof}

\bigskip

\begin{proof}
[Proof of Lemma \ref{LemIntTail1}]The second part assuming that $\bar{F}%
(\cdot)$ is regularly varying follows from Karamata's theorem. Now, for
non-regularly varying part, we simply note using L'Hopital's rule and Lemma
\ref{LemRateFun},
\[
\lim_{x\rightarrow\infty} \frac{\bar F(x)}{G(x)} = \lim_{x\rightarrow\infty}
\lambda(x) =0.
\]
The lower bound follows immediately. Again, using L'Hopital's rule, the upper
bound then follows from the fact that
\[
\lim_{x\rightarrow\infty} \frac{x\bar F(x)}{G(x)} =\lim_{x\rightarrow\infty}
\frac{x\lambda(x)\bar F(x)- \bar F(x)}{\bar F(x)} = \infty.
\]
The last step is thanks to Assumption B1.

\end{proof}

\bigskip

\begin{proof}
[Proof of Lemma \ref{LemCen}]This is a direct application of condition B2.
Indeed, if $x\geq b_{0}>0$ and $y\geq0$%
\[
\log\Lambda(x+y)-\log\Lambda(x)=\int_{x}^{x+y}\partial\log\Lambda\left(
t\right)  dt\leq\int_{x}^{x+y}\beta_{0}t^{-1}dt=\beta_{0}\log\left(
\frac{x+y}{x}\right)  ,
\]
which is equivalent to the statement of the lemma.
\end{proof}

\bigskip

\begin{proof}
[Proof of Lemma \ref{LemOvershoot}]Equivalently, we must show that for $x$
sufficiently large%
\[
a_{\ast}\geq\Lambda(x)-\Lambda\left(  x-x^{\alpha}\right)  ,
\]
where $\alpha=(1-\beta_{0})/2$. Now, note using Lemma \ref{LemCen} that%
\[
\Lambda(x)-\Lambda\left(  x-x^{\alpha}\right)  \leq\Lambda\left(  x-x^{\alpha
}\right)  \left(  \frac{\Lambda(x)}{\Lambda\left(  x-x^{\alpha}\right)
}-1\right)  \leq\Lambda\left(  x-x^{\alpha}\right)  \left(  \left(  \frac
{x}{x-x^{\alpha}}\right)  ^{\beta_{0}}-1\right)  .
\]
For all $x$ sufficiently large, using a Taylor expansion, the right hand side
is bounded by $\Lambda\left(  x-x^{\alpha}\right)  \left(  2\beta_{0}%
x^{\alpha-1}\right)  $. Consequently, once again applying Lemma \ref{LemCen}
we conclude that%
\[
\Lambda(x)-\Lambda\left(  x-x^{\alpha}\right)  \leq\Lambda\left(  x-x^{\alpha
}\right)  \left(  2\beta_{0}x^{\alpha-1}\right)  \leq4\beta_{0}\Lambda\left(
b_{0}\right)  x^{\beta_{0}-1 + \alpha}
\]
The right hand side goes to zero as $x\nearrow\infty$ given our selection of
$\alpha$ and therefore is less than $a_{\ast}$ for all $x$ sufficiently large
as required.
\end{proof}

\bigskip

\begin{proof}
[Proof of Lemma \ref{LemSubExp}]If Assumption A is satisfied then it is well
known that both $F$ and $G$ are subexponential. Let us then assume that B2
holds, and then we obtain $x\lambda\left(  x\right)  \leq\beta_{0}%
\Lambda\left(  x\right)  $ for all $x\geq b_{0}$ and $\beta_{0}\in\left(
0,1\right)  $. Applying Pitman's criterion (Proposition \ref{PropPitmanC}) and
the fact that (by Lemma \ref{LemRateFun} in particular $\lambda\left(
x\right)  =O\left(  1\right)  $ for $x\geq b_{0}$) it suffices to verify that%
\[
\int_{b_{0}}^{\infty}\exp\left(  x\lambda\left(  x\right)  -\Lambda\left(
x\right)  \right)  dx<\infty.
\]
Nevertheless, combining B1 and B2 we have that there exists $c\in\left(
0,\infty\right)  $ such that%
\[
\int_{b_{0}}^{\infty}\exp\left(  x\lambda\left(  x\right)  -\Lambda\left(
x\right)  \right)  dx\leq\int_{b_{0}}^{\infty}e^{\left(  \beta_{0}-1\right)
\Lambda\left(  x\right)  }dx\leq c\int_{b_{0}}^{\infty}x^{-2}dx<\infty
\]
and we conclude the lemma.

For the subexpontentiality of the integrated tail, it is sufficient to show
that
\[
\limsup_{x\rightarrow\infty} \frac{x\bar F(x)}{-G(x) \log G(x)} < 1,
\]
and apply the same analysis for the subexponentiality of $\bar F$. By
L'Hopital's rule (possibly on a subsequence),
\[
\limsup_{x\rightarrow\infty} \frac{x\bar F(x)}{-G(x) \log G(x)} \leq
\limsup_{x\rightarrow\infty} -\frac{x\lambda(x) - 1}{1+\log G(x)} \leq
\limsup_{x\rightarrow\infty} -\frac{x\lambda(x) - 1}{\log\varepsilon+ \log x
-\Lambda(x)}\leq\beta_{0}%
\]
The second inequality is due to Lemma \ref{LemIntTail1}. The last inequality
is from the fact that $\log x = o(\Lambda(x))$ and Assumptions B1 and B2.
$\bar F(x)/G(x)$ and $-\log G(x)$ are the hazard function and cumulative hazard function of
the integrated tail. The proof is completely analogous and therefore is omitted.
\end{proof}

\bigskip

\begin{proof}
[Proof of Lemma \ref{LemWeibull}] Given $\beta_0\in(0,1)$, one can always select $\sigma_1$ as indicated in the statement of the lemma. Note that there exists a $\delta>0$ such that for all
$\sigma_{1}\leq x\leq1-\sigma_{1}$
\[
x^{\beta_0}+(1-x)^{\beta_0}\geq1+\delta.
\]
So, by continuity and with $\sigma_1$ small enough, we can find $\sigma_{2}>0$ small
enough so that
\[
x^{\beta_0}+(1-x-\sigma_{1}/2)^{\beta_0}\geq1+\sigma_{2}.
\]
Therefore, we know that we can select
\[
a_{j}= a_{j-1}+\sigma_{1}/2,
\]
as long as $\sigma_{1}/2\leq a_{j-1}\leq1-\sigma_{1}/2$. Now select
$k=\lceil2(1-\sigma_{1})/\sigma_{1}\rceil$ and we have $a_{k}\geq1-\sigma
_{1}/2$.
\end{proof}

\section{Technical proofs in Section \ref{SecTV}}

\label{ApdTotal}

\begin{proof}
[Proof of Lemma \ref{LemDensDec}]First it is straightforward to verify
(\ref{decomp}) out of definition (\ref{Mix}).
By integrating both sides of \eqref{decomp}, it is also immediate to see
$$\int_{-\infty}^{\infty} q^*_s(x)dx =1.$$
Now, we just need to verify that
if $p\left(  s\right)  <1$ then $(1-p(s))q_{s}^{\ast}(x)\geq0$.
We concentrate on the case in which Assumption B
prevails (if Assumption A is in force the arguments carry over in very similar
forms). When $b-s > \eta_*$, using the definition of $q_{s}\left(  x\right)  $ given in Section
\ref{SubMixFam} we obtain%
\begin{align*}
q_{s}\left(  x\right)   &  =p_{\ast}f(x)\frac{I(x\leq c_{0})}{P(X\leq c_{0}%
)} +p_{\ast\ast}f_{\ast\ast}(x|s)+\sum_{j=1}^{k}p_{j}f_{j}(x|s) \\
&  =\frac{p_{\ast}f(x)}{P(X\leq c_{0}%
)}\\
&  +\frac{p_{\ast\ast}f(x)I(x>c_{k})}{P(X>c_{k})}-\frac{p_{\ast}f\left(
x\right)  I(x>c_{k})}{P\left(  X\leq c_{0}\right)  }\\
&  +\frac{p_{k}f(b-s-x)I(x\in(c_{k-1},c_{k}])}{P(X\in(b-s-c_{k},b-s-c_{k-1}%
])}-\frac{p_{\ast}f\left(  x\right)  I(x\in(c_{k-1},c_{k}])}{P\left(  X\leq
c_{0}\right)  }\\
&  +\sum_{j=1}^{k-1}\left(  \frac{p_{j}f(x)I(x\in(c_{j-1},c_{j}])}%
{P(X\in(c_{j-1},c_{j}])}-\frac{p_{\ast}f\left(  x\right)  I(x\in(c_{j-1}%
,c_{j}])}{P\left(  X\leq c_{0}\right)  }\right)  .
\end{align*}
Therefore,%
\begin{align}
\left(  1-p\left(  s\right)  \right)  q_{s}^{\ast}\left(  x\right)   &
=\frac{p_{\ast\ast}f(x)I(x>c_{k})}{P(X>c_{k})}-\frac{p_{\ast}f\left(
x\right)  I(x>c_{k})}{P\left(  X\leq c_{0}\right)  }\nonumber\\
&  +\frac{p_{k}f(b-s-x)I(x\in(c_{k-1},c_{k}])}{P(X\in(b-s-c_{k},b-s-c_{k-1}%
])}-\frac{p_{\ast}f\left(  x\right)  I(x\in(c_{k-1},c_{k}])}{P\left(  X\leq
c_{0}\right)  }\nonumber\\
&  +\sum_{j=1}^{k-1}\left(  \frac{p_{j}f(x)I(x\in(c_{j-1},c_{j}])}%
{P(X\in(c_{j-1},c_{j}])}-\frac{p_{\ast}f\left(  x\right)  I(x\in(c_{j-1}%
,c_{j}])}{P\left(  X\leq c_{0}\right)  }\right)  .\label{q}%
\end{align}
To verify that $(1-p(s))  q_{s}^{\ast}(x)  \geq0$, the most interesting part involves the second line in the above display
corresponding to the interval $x\in(c_{k-1},c_{k}]$. The reasoning for the
rest of the pieces is similar and therefore is omitted. On the
interval $(c_{k-1},c_{k}]$ we have that $b-s-x\leq x$ assuming that
$b-s\geq\eta_{\ast}$ and $\eta_{\ast}$ is sufficiently large.
Since $f\left(  \cdot\right)  $ is eventually decreasing (a consequence of
Assumption B3), then
\[
f\left(  b-s-x\right)  \geq f\left(  x\right),
\]
when $x\in(c_{k-1},c_{k}]$. Consequently
\begin{align*}
&  \frac{p_{k}f(b-s-x)I(x\in(c_{k-1},c_{k}])}{P(X\in(b-s-c_{k},b-s-c_{k-1}%
])}-\frac{p_{\ast}f\left(  x\right)  I(x\in(c_{k-1},c_{k}])}{P\left(  X\leq
c_{0}\right)  }\\
&  \geq\frac{p_{k}f\left(  x\right)  I(x\in(c_{k-1},c_{k}])}{P(X\in
(b-s-c_{k},b-s-c_{k-1}])}-\frac{p_{\ast}f\left(  x\right)  I(x\in
(c_{k-1},c_{k}])}{P\left(  X\leq c_{0}\right)  }.
\end{align*}
Further, we have that $p_{k}=\varepsilon^{2}p_{\ast\ast}$ decreases to zero at
most linearly in $(b-s)^{-1}$, whereas $P(X\in(b-s-c_{k},b-s-c_{k-1}])$ goes to zero
faster than any linear function of  $(b-s)^{-1}$. Therefore, $\left(  1-p\left(  s\right)  \right)
q_{s}^{\ast}\left(  x\right)  I\left(  x\in(c_{k-1},c_{k}]\right)  \geq0$. The
remaining pieces in (\ref{q}) are handled similarly.
\end{proof}

\bigskip

\begin{proof}
[Proof of Lemma \ref{LemN}]Note that%
\begin{equation}
Q^{\ast}\left(  N_{b}>kb\right)  =E^{Q^{\ast}}\left(  \prod_{j=0}^{\left\lceil
kb\right\rceil }p(\tilde{S}_{j})\right)  , \label{Qs0}%
\end{equation}
where $p(s)$ is defined in \eqref{DefpsD}. In addition, for some $\varepsilon>0$,%
\begin{align}
E^{Q^{\ast}}\left(  \prod_{j=0}^{\left\lceil kb\right\rceil }p(\tilde
{S}_{j})\right)   &  \leq E^{Q^*}\left(  \prod_{j=0}^{\left\lceil kb\right\rceil
}p(\tilde{S}_{j})I(|\tilde{S}_{j}-\mu j|\leq\varepsilon\max\{j,b\})\right)
\nonumber\\
&  +Q^{\ast}\left(  \sup_{j=1}^{\left\lceil kb\right\rceil }|\tilde{S}_{j}-\mu
j|-\varepsilon\max\{j,b\}>0\right)  . \label{Qs1}%
\end{align}
Notice that for any $\varepsilon>0$,%
\[
\lim_{b\rightarrow\infty}Q^{\ast}\left(  \sup_{j=1}^{\left\lceil
kb\right\rceil }[|\tilde{S}_{j}-\mu j|-\varepsilon\max\{j,b\}]>0\right)  =0.
\]
Then, for some $K$ sufficiently large (using an argument similar to that given
in the proof of Proposition \ref{PropCoup}) we conclude
\[
E^{Q^*}\left(  \prod_{j=0}^{\left\lceil kb\right\rceil }p(\tilde{S}%
_{j})I(|\tilde{S}_{j}-\mu j|\leq\varepsilon\max\{j,b\})\right)  \leq
Kk^{-\varepsilon_0}.
\]
for some $\varepsilon_0$ small enough. This is because $1-p(s) = (1+o(1))p_{**}$ as $b-s\rightarrow \infty$ and $\varepsilon \rightarrow 0$.
Thereby, we conclude the proof applying the previous two estimates into
(\ref{Qs0}) and (\ref{Qs1}).
\end{proof}

\begin{proof}
[Proof of Lemma \ref{LemJump}]Let%
\[
\int_{b-s}^{\infty}q_{s}^{\ast}(x)dx=R\left(  s\right)  .
\]
Note that for $b-s>\eta_{\ast}$ we have that
\begin{equation}
R\left(  s\right)  =O\left(  \varepsilon\right)  +e^{-a_{\ast\ast}}.
\label{ObsAux1}%
\end{equation}

Let

\[
\tau_{b}^{\prime}=\inf\{n\geq1: \tilde S_{n}\geq b\}.
\]

Now observe that%
\begin{align*}
Q^{\ast}\left(  \tau_{b}=N_{b}\right)   &  =\sum_{k=1}^{\infty}Q^{\ast}%
(N_{b}=k,\hat{S}_{k}>b,\tau_{b}>k-1)\\
&  \geq\sum_{k=1}^{\infty}Q^{\ast}(N_{b}=k,\hat{S}_{k}>b,\tau_{b-\eta_{\ast}%
}^{\prime}>k-1).
\end{align*}
Because of (\ref{ObsAux1}) we obtain that%
\begin{align*}
\sum_{k=1}^{\infty}Q^{\ast}(N_{b}  &  =k,\hat{S}_{k}>b,\tau_{b-\eta_{\ast}%
}^{\prime}>k-1)\\
&  \geq(O(\varepsilon)+e^{-a_{\ast\ast}})\sum_{k=1}^{\infty}Q^{\ast}%
(N_{b}=k,\tau_{b-\eta_{\ast}}^{\prime}>k-1)\\
&  =(O(\varepsilon)+e^{-a_{\ast\ast}})Q^{\ast}(\tau_{b-\eta_{\ast}}^{\prime
}>N_{b}-1,N_{b}<\infty)\\
&  \geq(O(\varepsilon)+e^{-a_{\ast\ast}}+o(1))Q^{\ast}(\tau_{b-\eta_{\ast}%
}^{\prime}=\infty).
\end{align*}
The term $o(1)\rightarrow0$ as $b\rightarrow\infty$ comes from Lemma
\ref{LemN} which shows that $Q^{\ast}(N_{b}=\infty)=o\left(  1\right)  $ as
$b\rightarrow\infty$. Finally, we observe%
\[
Q^{\ast}(\tau_{b-\eta_{\ast}}^{\prime}=\infty)=1-u(b-\eta_{\ast}%
)\rightarrow1,
\]
as $b\rightarrow\infty$. The conclusion of this lemma follows.
\end{proof}

\bigskip

\begin{proof}
[Proof of Proposition \ref{PropCoup}]For $\delta_{b}=1/\log b$, define
\[
B_{b}=\{\tilde S:|\tilde S_{j}-j\mu|\leq\max(\delta_{b}^{-1},\delta
_{b}j),1\leq j\leq t a(b)\}.
\]
It is clear that $\lim _{b\rightarrow \infty}P(B_b)=1$.

\paragraph{If $F$ is regularly varying,}

note that $1-p(s) = (1+o(1))p_{**}$ as $b-s\rightarrow \infty $, $\varepsilon \rightarrow 0$. For all $\tilde{S}\in B_{b}$
\[
Q^{\ast}\left(  N_{b}>ta(b)|\tilde S\right)  =\prod_{j=0}%
^{\left\lfloor ta(b)\right\rfloor } p(\tilde S_j)
=(1+o(1))\exp\left\{  -\sum
_{j=0}^{\left\lfloor ta(b)\right\rfloor }2\theta\frac{\bar{F}(b+j\left\vert
\mu\right\vert )}{G(b+j\left\vert \mu\right\vert )}\right\}  .
\]
By Karamata's theorem we have that%
\[
\sum_{j=0}^{\left\lfloor ta(b)\right\rfloor }2\theta\frac{\bar{F}%
(b+j\left\vert \mu\right\vert )}{G(b+j\left\vert \mu\right\vert )}%
\rightarrow\frac{2\theta(\iota-1)}{| \mu|}\log\left(  1+\frac{|\mu|t}{\iota
-1}\right)  .
\]

\paragraph{If Assumptions B1-B4 hold,}
We clearly have that
\begin{align}
a(x) &  =\frac{G(x)}{\bar{F}(x)}\label{AA}\\
&  =\int_{0}^{\infty}P(X>x+t|X>x)dt\nonumber\\
&  =\frac{1}{\lambda(x)}\int_{0}^{\infty}P(X>x+t/\lambda(x)|X>x)dt.\nonumber
\end{align}
Now we can invoke Assumption B4 together with the dominated convergence
theorem to conclude that
\[
\int_{0}^{\infty}P(X>x+t/\lambda(x)|X>x)dt\longrightarrow\int_{0}^{\infty}%
\exp\left(  -t\right)  dt=1
\]
as $x\longrightarrow\infty$. In addition, by the fundamental theorem of
calculus we have that%
\[
\Lambda\left(  x+y/\lambda\left(  x\right)  \right)  -\Lambda\left(  x\right)
=\frac{y}{\lambda\left(  x\right)  }\int_{0}^{1}\lambda\left(  x+yu/\lambda
\left(  x\right)  \right)  du
\]
and, in view of this representation, Assumption B4 is equivalent to stating
that for each $K\in\left(  0,\infty\right)  $%
\begin{equation}
\lim_{x\rightarrow\infty}\sup_{0\leq y\leq K}\left\vert \int_{0}%
^{y/\lambda\left(  x\right)  }\lambda\left(  x+z\right)  dz-y\right\vert
=\lim_{x\rightarrow\infty}\sup_{0\leq y\leq K}\left\vert \int_{0}^{ya\left(
x\right)  }\frac{1}{a\left(  x+z\right)  }dz-y\right\vert =0.\label{BB}%
\end{equation}
Observe that, since $\lambda\left(  \cdot\right)  $ is eventually
non-increasing,%
\[
\sum_{j=0}^{\left\lfloor t/\lambda\left(  b\right)  \right\rfloor }%
\lambda\left(  b+\left(  j+1\right)  \left\vert \mu\right\vert \right)
\leq\int_{0}^{t/\lambda\left(  b\right)  }\lambda\left(  b+x\left\vert
\mu\right\vert \right)  dx\leq\sum_{j=0}^{\left\lfloor t/\lambda\left(
b\right)  \right\rfloor }\lambda\left(  b+j\left\vert \mu\right\vert \right)
.
\]
We then conclude that%
\[
0\leq \int_{0}^{t/\lambda\left(  b\right)  }\lambda\left(
b+x\left\vert \mu\right\vert \right)  dx-\sum_{j=0}^{\left\lfloor
t/\lambda\left(  b\right)  \right\rfloor }\lambda\left(  b+\left(  j+1\right)
\left\vert \mu\right\vert \right)  \leq\lambda\left(  b\right)
\longrightarrow0
\]
as $b\nearrow\infty$. Therefore, applying (\ref{AA}) and (\ref{BB}) we
conclude that%
\[
\lim_{b\rightarrow\infty}\sum_{j=0}^{\left\lfloor ta(b)\right\rfloor }%
\frac{\bar{F}(b+j|\mu|)}{G(b+j|\mu|)}=\lim_{b\rightarrow\infty}\int
_{0}^{t/\lambda\left(  b\right)  }\lambda\left(  b+x\left\vert \mu\right\vert
\right)  dx=t
\]
as $b\nearrow\infty$ and consequently we have that for all $\tilde{S}\in
B_{b}$
\[
Q^{\ast}\left(  N_{b}>ta(b)|\tilde{S}\right)  =(1+o(1))\exp\left\{
-\sum_{j=0}^{\left\lfloor ta(b)\right\rfloor }2\theta\frac{\bar{F}(b+j|\mu
|)}{G(b+j|\mu|)}\right\}  .
\]
We then conclude that%
\[
\lim_{b\rightarrow\infty}Q^{\ast}\left(  N_{b}>ta(b)|\tilde{S}\right)
=e^{-2\theta t}.
\]
\end{proof}

\bibliographystyle{plain}
\bibliography{bibprob,bibstat,mod_prob}

\end{document}